\RequirePackage{ifpdf} 
\ifpdf
\documentclass[11pt,pdftex]{article}
\else
\documentclass[11pt,dvips]{article}
\fi

\ifpdf
\usepackage{lmodern}
\fi
\usepackage[bookmarks, 
colorlinks=false, backref,plainpages=false]{hyperref}
\usepackage[T1]{fontenc}
\usepackage[latin1]{inputenc}
\usepackage[english]{babel}
\usepackage{graphicx}
\usepackage{url}
\usepackage{graphics}
\usepackage{epsfig,color}
\usepackage{amsmath,amssymb}
\usepackage{amsthm}
 \usepackage{relsize}
\renewcommand{\theequation}{\thesection.\arabic{equation}}

\newtheorem{thm}{Theorem}[section]

\newtheorem{rem}[thm]{Remark}

\begin{document}
\newcommand{\BX}{{\bf X}}
\newcommand{\cv}{{\cal V}}
\newcommand{\cW}{{\cal W}}
\newcommand{\co}{{\cal O}}

\renewcommand{\theequation}{\thesection.\arabic{equation}}
\def\@eqnnum{{\reset@font\rm (\theequation)}}

\def\abstract{
\advance \rightskip by 10mm
\advance \leftskip by 10mm
\vspace{-0.8em}
\noindent
\small{\bf Abstract.}
}
\def\endabstract{\par\normalsize\rm}

\def\Xint#1{\mathchoice
{\XXint\displaystyle\textstyle{#1}}%
{\XXint\textstyle\scriptstyle{#1}}%
{\XXint\scriptstyle\scriptscriptstyle{#1}}%
{\XXint\scriptscriptstyle\scriptscriptstyle{#1}}%
\!\int}
\def\XXint#1#2#3{{\setbox0=\hbox{$#1{#2#3}{\int}$}
\vcenter{\hbox{$#2#3$}}\kern-.5\wd0}}
\def\ddashint{\Xint=}
\def\dashint{\Xint-}

\def\a{\alpha}
\def\b{\beta}
\def\d{\delta}\def\D{\Delta}
\def\e{\epsilon}
\def\g{\gamma}\def\G{\Gamma}
\def\k{\kappa}
\def\lam{\lambda}\def\Lam{\Lambda}
\renewcommand\o{\omega}\renewcommand\O{\Omega}
\def\s{\sigma}\def\S{\Sigma}
\renewcommand\t{\theta}\def\vt{\vartheta}
\newcommand{\vphi}{\varphi}
\def\z{\zeta}

\newcommand{\tsigma}{\tilde{\s}}
\newcommand{\tbsigma}{\tilde{\bsigma}}
\def\te{\tilde{\e}}
\def\tu{\tilde{u}}

\newcommand{\bchi}{\mbox{\boldmath$\chi$}}
\newcommand{\bdelta}{\mbox{\boldmath$\delta$}}
\newcommand{\bepsilon}{\mbox{\boldmath$\epsilon$}}
\newcommand{\bfeta}{\mbox{\boldmath$\eta$}}
\newcommand{\bgamma}{\mbox{\boldmath$\gamma$}}
\newcommand{\bomega}{\mbox{\boldmath$\omega$}}
\newcommand{\bvphi}{\mbox{\boldmath$\varphi$}}
\newcommand{\bphi}{\mbox{\boldmath$\phi$}}
\newcommand{\bPhi}{\mbox{\boldmath$\Phi$}}
\newcommand{\bpsi}{\mbox{\boldmath$\psi$}}
\newcommand{\bPsi}{\mbox{\boldmath$\Psi$}}
\newcommand{\bsigma}{\mbox{\boldmath$\sigma$}}
\newcommand{\btau}{\mbox{\boldmath$\tau$}}
\newcommand{\bxi}{\mbox{\boldmath$\xi$}}
\newcommand{\brho}{\mbox{\boldmath$\rho$}}
\newcommand{\bbeta}{\mbox{\boldmath$\beta$}}
\newcommand{\bzeta}{\mbox{\boldmath$\zeta$}}

\def\bk{\boldsymbol{\kappa}}
\def\bmu{\boldsymbol\mu}
\def\bxi{\boldsymbol{\xi}}
\def\bz{\boldsymbol{\zeta}}

\def\ba{{\bf a}}
\def\bb{{\bf b}}
\def\bc{{\bf c}}
\def\be{{\bf e}}
\def\bff{{\bf f}}
\def\bg{{\bf g}}
\def\bn{{\bf n}}
\def\bp{{\bf p}}
\def\bq{{\bf q}}
\def\bs{{\bf s}}
\def\bt{{\bf t}}
\def\bu{{\bf u}}
\def\bv{{\bf v}}
\def\bw{{\bf w}}
\def\bx{{\bf x}}
\def\by{{\bf y}}
\def\bzz{{\bf z}}

\def\bD{{\bf D}}
\def\bE{{\bf E}}
\def\bF{{\bf F}}
\def\bH{{\bf H}}
\def\bJ{{\bf J}}
\def\bV{{\bf V}}
\def\bU{{\bf U}}
\def\bW{{\bf W}}
\def\bX{{\bf X}}
\def\bY{{\bf Y}}

\def\cA{{\cal A}}
\def\cC{{\cal C}}
\def\cD{{\cal D}}
\def\cE{{\cal E}}
\def\cF{{\cal F}}
\def\cG{{\cal G}}
\def\cI{{\cal I}}
\def\cJ{{\cal J}}
\def\cK{{\cal K}}
\def\cL{{\cal L}}
\def\cO{{\cal O}}
\def\cP{{\cal P}}
\def\cQ{{\cal Q}}
\def\cR{{\cal R}}
\def\cS{{\cal \Sigma}}
\def\cT{{\cal T}}
\def\cU{{\cal U}}
\def\cV{{\cal V}}

\def\scT{{_\cT}}
\def\sD{{_D}}
\def\sE{{_E}}
\def\sF{{_F}}
\def\sFz{{_{F_z}}}
\def\sK{{_K}}
\def\sI{{_I}}
\def\sb{{_b}}
\def\sN{{_N}}

\def\curl{{{\bf curl} \ }}
\def\rot{{\mbox{rot}\ }}
\def\BPI{{\bf \Pi}}

\def\cth{\cT_h}
\def\ctH{\cT_H}

\def\tJ{\tilde{\J}}

\def\hK{\widehat{K}}
\def\hx{\widehat{x}}
\def\hy{\widehat{y}}
\def\bhv{\widehat{\bv}}

\def\l{\ell}
\def\bl{\boldsymbol{\ell}}
\def\col{\colon}
\def\f12{\frac12}
\def\dfrac{\displaystyle\frac}
\def\dint{\displaystyle\int}
\def\nab{\nabla}
\def\p{\partial}
\def\sm{\setminus}
\def\dsum{\displaystyle\sum}
\newcommand{\pp}[2]{\frac{\partial {#1}}{\partial {#2}}}
\def\bzero{{\bf 0}}

\def\divv{\nab\cdot}
\def\divx{\nab_x\cdot}
\def\divtx{\nab_{t,x}\cdot}
\def\nabx{\nab_x}

\newcommand{\curlt}{{\nabla \times}}
\newcommand{\gperp}{\nabla^{\perp}}
\newcommand{\gradt}{\nabla\cdot}

\def\forallqq{\quad\forall\,}
\def\aph{A^{1/2}}
\def\amh{A^{-1/2}}

\def\osc{{\rm osc \, }}

\def\Im{{\rm Im}}
\newcommand{\tr}{{\rm tr}}
\def\divvr{{\rm div}}
\def\curllr{{\rm curl}}
\def\curll{{\rm curl}}
\def\curl{{\bf curl}}
\newcommand{\bgrad}{{\bf grad}}
\newcommand\diam{\mathrm{diam\,}}
\renewcommand\Im{\mathrm{Im\,}}
\def\Span{\mbox{Span}}
\def\supp{\mbox{supp\,}}
\newcommand{\trace}{{\rm trace}}

\newcommand{\tri}{|\!|\!|}
\newcommand{\ljump}{\lbrack\!\lbrack}
\newcommand{\rjump}{\rbrack\!\rbrack}
\newcommand{\bdm}{\begin{displaymath}}
\newcommand{\edm}{\end{displaymath}}
\newcommand{\beq}{\begin{equation}}
\newcommand{\eeq}{\end{equation}}
\newcommand{\beqa}{\begin{eqnarray}}
\newcommand{\eeqa}{\end{eqnarray}}
\newcommand{\beqas}{\begin{eqnarray*}}
\newcommand{\eeqas}{\end{eqnarray*}}
\newcommand{\ul}{\underline}
\newcommand{\wh}{\widehat}
\newcommand{\la}{\langle}
\newcommand{\ra}{\rangle}

\newcommand{\Lt}{L^2(\Omega)}
\newcommand{\Lts}{L^2(\Omega)^2}
\newcommand{\Ltc}{L^2(\Omega)^3}
\newcommand{\Ho}{H^1(\Omega)}
\newcommand{\Hoh}{H^1(\wh{\Omega})}
\newcommand{\Hoi}{H^1(\Omega_i)}
\newcommand{\Hos}{H^1(\Omega)^2}
\newcommand{\Hoc}{H^1(\Omega)^3}
\newcommand{\Hoch}{H^1(\wh{\Omega})^3}
\newcommand{\Hoci}{H^1(\Omega_i)^3}
\newcommand{\Hoz}{H^1_0(\Omega)}
\newcommand{\Ht}{H^2(\Omega)}
\newcommand{\Hti}{H^2(\Omega_i)}
\newcommand{\Hts}{H^2(\Omega)^2}
\newcommand{\Htc}{H^2(\Omega)^3}
\newcommand{\Htz}{H^0(\Omega)}
\newcommand{\Hh}{H^{1/2}(\Gamma)}
\newcommand{\Hhi}{H^{1/2}(\Gamma_i)}
\newcommand{\Hmh}{H^{-1/2}(\Gamma)}
\newcommand{\Hdiv}{H(\divvr;\,\Omega)}
\newcommand{\Hdivh}{H(\divv;\,\wh \Omega)}
\newcommand{\hcurl}{H(\curl\,A;\,\Omega)}
\newcommand{\Hcurl}{H(\curll\,A;\,\Omega)}
\newcommand{\Hcrl}{H(\curll\,;\,\Omega)}
\newcommand{\hcrl}{H(\curl\,;\,\Omega)}
\newcommand{\Hcrlh}{H(\curll\,;\,\wh\Omega)}
\newcommand{\hcrlh}{H(\curl\,;\,\wh\Omega)}
\newcommand{\Wdiv}{\BW_0(\mbox{\divv}\,;\,\Omega)}
\newcommand{\Wcurl}{\BW_0(\mbox{\curl}\,A;\,\Omega)}
\newcommand{\WcrossV}{\BW \times V}

\def\calS{{\cal S}}
\def\calT{{\cal T}}
\def\cB{{\cal B}}
\def\cH{{\cal H}}
\def\ba{{\mathbf{a}}}
\def\cM{{\cal M}}
\def\cN{{\cal N}}

\def\bE{{\bf E}}
\def\bS{{\bf S}}
\def\br{{\bf r}}
\def\bW{{\bf W}}
\def\bLambda{{\bf \Lambda}}

\newcommand{\lJump}{[\![}
\newcommand{\rJump}{]\!]}
\newcommand{\jump}[1]{[\![ #1]\!]}

\newcommand{\sd}{\bsigma^{\Delta}}
\newcommand{\st}{\tilde{\bsigma}}
\newcommand{\sh}{\hat{\bsigma}}
\newcommand{\rd}{\brho^{\Delta}}

\newcommand{\WH}{W\!H}
\newcommand{\NE}{N\!E}

\newcommand{\ND}{N\!D}
\newcommand{\RT}{RT}
\newcommand{\ZZ}{Z\!Z}
\newcommand{\BDM}{B\!D\!M}

\newcommand{\sRT}{{_{RT}}}
\newcommand{\sBDM}{{_{BDM}}}
\newcommand{\sWH}{{_{WH}}}
\newcommand{\sND}{{_{ND}}}
\newcommand{\sV}{_\cV}

\newcommand{\dd}{\st{{\mathbf d}}}
\newcommand{\C}{\rm I\kern-.5emC}
\newcommand{\R}{\rm I\kern-.19emR}
\newcommand{\W}{{\mathbf W}}
\def\3bar{{|\hspace{-.02in}|\hspace{-.02in}|}}
\newcommand{\A}{{\mathcal A}}

\title {Improved ZZ A Posteriori Error Estimators for Diffusion Problems:\\
Conforming Linear Elements}
\author{
Zhiqiang Cai\thanks{
Department of Mathematics, Purdue University, 150 N. University
Street, West Lafayette, IN 47907-2067, \{caiz, he75\}@purdue.edu.
This work was supported in part by the National Science Foundation
under grants DMS-1217081 and DMS-1522707.}
\and Cuiyu He\footnotemark[1]
 \and Shun Zhang\thanks{Department of Mathematics, 
City University of Hong Kong, Hong Kong SAR, China,
shun.zhang@cityu.edu.hk.
This work was supported in part
by Hong Kong Research Grants Council under 
the GRF Grant Project No. 11303914, CityU 9042090.}
}
 \date{\today}
 \maketitle

\begin{abstract}
In \cite{CaZh:09}, we introduced and analyzed an 
improved Zienkiewicz-Zhu (ZZ) estimator 
for the conforming linear finite element approximation to 
elliptic interface problems. The estimator is based on 
the piecewise ``constant'' flux recovery in the $H(\divvr;\O)$ 
conforming finite element space.
This paper extends the results of \cite{CaZh:09} to 
diffusion problems with full diffusion tensor 
and to the flux recovery both in piecewise constant 
and piecewise linear $H(\divvr)$ space. 
\end{abstract}

\section{Introduction}\label{intro}
\setcounter{equation}{0}

A posteriori error estimation for finite element methods has been
extensively studied for the past four decades (see, e.g., books by
Ainsworth and Oden \cite{AiOd:00},
Babu\v{s}ka and Strouboulis \cite{BaSt:01}, and Verf\"urth \cite{Ver:13} and references therein).
Due to easy implementation, generality, and ability to produce quite accurate estimation, 
the Zienkiewicz-Zhu (ZZ) recovery-based error estimator \cite{ZiZh:87} 
has been widely adapted in engineering practice and has been the subject of mathematical study
(e.g., \cite{BaXu:03, BaXuZh:07, CaZh:10c, CaBa:02b, CaBaKl:01, HoScWaWi:01, NaZh:05, Ro:94, ScWa:04, 
YaZh:01, Zh:01, Zh:07, ZhZh:95, ZiZh:92}).
By first recovering 
a gradient (flux) in the conforming $C^0$
linear vector finite element space from the numerical gradient (flux),
the ZZ estimator is defined as the $L^2$ norm of the difference 
between the recovered and the numerical gradients/fluxes. 

Despite popularity of the ZZ estimator, 
it is also well known (see, e.g., \cite{Ova:06a, Ova:06b})  that adaptive mesh 
refinement (AMR) algorithms using the ZZ estimator are not efficient to reduce global error
for non-smooth problems, e.g., interface problems. 
This is because they over-refine regions where there are only small errors.
By exploring the mathematical structure of the underlying problem and the characteristics of finite element 
approximations, in \cite{CaZh:09, CaZh:10a} we identified that this failure of the ZZ estimator is 
caused by using a continuous function (the recovered gradient (flux))
to approximate a discontinuous one (the true gradient (flux)) in the
recovery procedure. Therefore, to fix this structural failure, we should recover the gradient (flux) in proper finite element spaces.
More specifically, for the conforming linear finite element approximation to the interface problem
we recovered the flux in the $H(\divvr;\O)$ conforming finite element space.
It was shown in \cite{CaZh:09} that the resulting implicit and explicit error estimators are not only reliable but also efficient.
 Moreover, the estimators are robust with respect to the jump of 
the coefficients.

In \cite{CaZh:09}, the implicit error estimator requires solution of a global $L^2$ minimization problem, 
and the explicit error estimator uses a simple edge average. This averaging 
approach of the explicit estimator is limited to the Raviart-Thomas ($RT$) \cite{BoBrFo:13} element
of the lowest order, i.e., the piecewise ``constant'' vector.
In this paper, we introduce a general approach of
constructing explicit flux recoveries using either the piecewise ``constant'' vector or the piecewise linear
vector (the Brezzi-Douglas-Marini (BDM) \cite{BoBrFo:13} element of the lowest order)
for the diffusion problem with the full diffusion coefficient tensor. 
With the recovered fluxes, the improved ZZ estimators are defined as a weighted $L^2$ norm of 
the difference between the recovered and the numerical fluxes. These estimators are 
theoretically shown to be locally efficient and globally reliable. 
Moreover, when the diffusion coefficient is piecewise constant scalar and
its distribution is locally quasi-monotone, these estimators are robust
with respect to the size of jumps. For a benchmark test problem, whose coefficient
is not locally quasi-monotone, numerical results also show the robustness of the estimators. 

The paper is organized as follows. Section 2 describes the diffusion
problem, variational form, and conforming finite element approximation. 
Section 3 describes the a posteriori error estimators of the ZZ type and two 
counter-examples.
Two explicit flux recoveries and their corresponding improved ZZ estimators 
are introduced in section 4. 
Efficiency and reliability of those estimators
are established in section 5.  Section 6 is devoted to explicit formulas of the recovered fluxes, the indicators,
and the estimators. Finally, section 7
presents numerical results on the Kellogg's benchmark test problem.

\section{Finite Element Approximation to Diffusion Problem}\label{problems}
\setcounter{equation}{0}

Let $\O$ be a bounded polygonal domain in $\Re^d$ with $d=2$ or $3$,
with boundary $\p \O = \overline{\Gamma}_\sD \cup\overline{ \Gamma}_\sN$,
$\Gamma_\sD\cap \Gamma_\sN = \emptyset$,  and $\mbox{meas}_{d-1}\,(\Gamma_{_D})\not= 0$,
and let $\bn$ be the outward unit vector normal to the boundary.
Consider diffusion equation
\begin{equation}\label{pde}
    -\gradt (A(x) \nabla u)  =  f   \quad\mbox{in} \,\,  \O
\end{equation}
with boundary conditions
\beq\label{bc}
    -A\nabla u \cdot \bn = g_\sN \quad\mbox{on} \,\, \Gamma_\sN \quad
    \mbox{and}\quad u = g_\sD \quad\mbox{on}\,\, \Gamma_\sD,
\eeq 
where the $\gradt$ and $\nabla$ are the divergence and gradient operators, respectively, and $f \in L^2(\O)$.
In this paper, we consider only simplicial  elements.
Let $\cT=\{K\}$ be a regular triangulation of the domain $\O$,
and denote by $h_\sK$ the diameter of the element $K$.
For simplicity of presentation,  assume that $g_\sD$  and
$g_\sN$ are piecewise affine functions and constants, respectively,
and that $A$ is a piecewise constant matrix that is symmetric, positive definite.

Let
 \[
 H^1_{_{g,D}}(\O)=\{v\in H^1(\O)\,:\, v=g_\sD\mbox{ on }\Gamma_\sD\}
 \quad\mbox{and}\quad
 H^1_{_D}(\O)=
 H^1_{_{0,D}}(\O).
 \]
Then the corresponding variational problem is to  find $u \in
H^1_{_{g,D}}(\O)$ such that
 \beq \label{vp}
    a(u,\,v)\equiv (A\nabla u, \nabla v)
    = (f, v) - ( g_\sN, v )_{\Gamma_\sN}\equiv  f(v)
    \quad \forall \; v\in H^1_{_D}(\O),
 \eeq
where $(\cdot, \cdot)_{\omega}$ is the $L^2$ inner product on the domain $\o$.
The subscript $\omega$ is omitted when $\o=\O$.

For each $K\in\cT$, let
$P_k(K)$ be the space of polynomials of degree less than or equal to $k$. 
Denote the linear 
conforming finite element space
 \cite{Cia:78} associated with the
triangulation $\cT$ by
 \[
 \calS = \{v\in H^1(\O)\,:\,v|_K\in P_1(K)\quad\forall\,\,K\in\cT\}.
\]
Let
 \[
 \calS_{g,\sD}= \{v\in \calS :\,v=g_\sD \,\,\mbox{on}\, \, \Gamma_\sD\}
 \quad \mbox{and} \quad
 \calS_{\sD}=\calS_{_{0,D}}
\]
Then the conforming finite element approximation is to seek $u_{_\cT} \in
\calS_{g,\sD}$ such that
\begin{eqnarray} \label{fem}
  ( A\nabla u_{_\cT} ,\, \nabla v) &=& f(v)  \qquad \,\forall\, v\in
  \calS_{\sD}.
\end{eqnarray}

\section{ZZ Error Estimator and Counter Examples}
\setcounter{equation}{0}

Let $\tilde{u}_{_\cT}\in \calS_{g,\sD}$ be an approximation to the finite element solution 
$u_{_\cT}\in \calS_{g,\sD}$ 
of (\ref{fem}). 
Denote the numerical gradient and the numerical flux by
 \beq\label{numer_gradient_flux}
 \tilde{\brho}_{_\cT} = \nabla \tilde{u}_{_\cT} 
 \quad\mbox{and}\quad
 \tilde{\bsigma}_{_\cT} = - A\, \nabla \tilde{u}_{_\cT},
 \eeq
respectively, which are piecewise constant vectors in $\Re^d$ with respect to the triangulation $\cT$.
It is common in engineering practice to smooth the piecewise constant  
gradient or flux in a post-process so that they are continuous. 
More precisely, denote by $\cN$ the set of all vertices of the triangulation $\cT$. For each vertex $z\in\cN$,
denote by $\omega_z$ the union of elements having $z$ as the common vertex. 
The recovered (smoothed) gradient or flux are defined as follows: for $\btau = \tilde{\brho}_{_\cT}$
or $\tilde{\bsigma}_{_\cT}$
 \beq\label{average}
 G(\btau) \in \calS^d\subset C^0(\Omega)^d \quad\mbox{with nodal values}\quad
 G(\btau) (z) =\dfrac{1}{|\omega_z|}\int_{\omega_z} \btau \,dx
 \quad\forall\,\, z\in \cN,
 \eeq
where $|\omega_z|=\mbox{meas}_d(\o_z)$.  
 There are many post-processing,
recovery techniques (see survey article \cite{Zh:07} by Zhang and references therein).
With the recovered gradient (flux), the ZZ error indicator and estimator are defined as follows:
 \[
 \xi_{_{ZZ,K}} = \|G(\btau) - \btau\|_{0,K}
 \quad\forall\,\, K\in\cT
  \quad\mbox{and}\quad 
  \xi_{_{ZZ}} = \|G(\btau) - \btau \|_{0,\O}
 \]
for $\btau = \tilde{\brho}_{_\cT}$ or $\tilde{\bsigma}_{_\cT}$, respectively.
 
Despite many attractive features of the ZZ error estimator, it is well known (see Figures in section~7  that 
adaptive mesh refinement (AMR) algorithms using the above ZZ estimator are not efficient to reduce global error
for interface problems. In this section, we demonstrate this failure of the ZZ estimator by two counter-examples 
for which the finite element solution is exact while the ZZ indicators along interface could be arbitrarily large. 

The first example 
is a one-dimensional interface problem defined on the domain $\Omega=(0,\,1)$ with the Dirichlet boundary 
condition:
 \[
 u(0)=0 
 \quad\mbox{and}\quad 
 u(1)=(k+1)/2
 \]
for an arbitrary constant $k>1$ and with piecewise constant diffusion coefficient 
 \[
 A = 1\quad\mbox{in }\, (0,\,1/2)
 \quad\mbox{and}\quad
 A=k >0 \quad\mbox{in }\, (1/2,\,1). 
 \]
 The exact solution and its derivative of this example are 
piecewise linear and piecewise constant functions, respectively, 
depicted in Figures 1 and 2 and given by 
\[
u = \left\{\begin{array}{ll} k\,x \quad &x\in (0,\,1/2\,],\\[2mm]
 x+\dfrac{k-1}{2}& x\in(1/2,\,1),
\end{array} \right.
\quad \mbox{and}\quad
u' = \left\{\begin{array}{ll} k \quad &x\in (0,\,1/2),\\[2mm]
 1 & x\in(1/2,\,1).
\end{array} \right.
\]

\begin{figure}[ht]
    \hfill
    \begin{minipage}[!htbp]{0.48\textwidth}
        \centering
        \includegraphics[width=0.85\textwidth, angle=0]{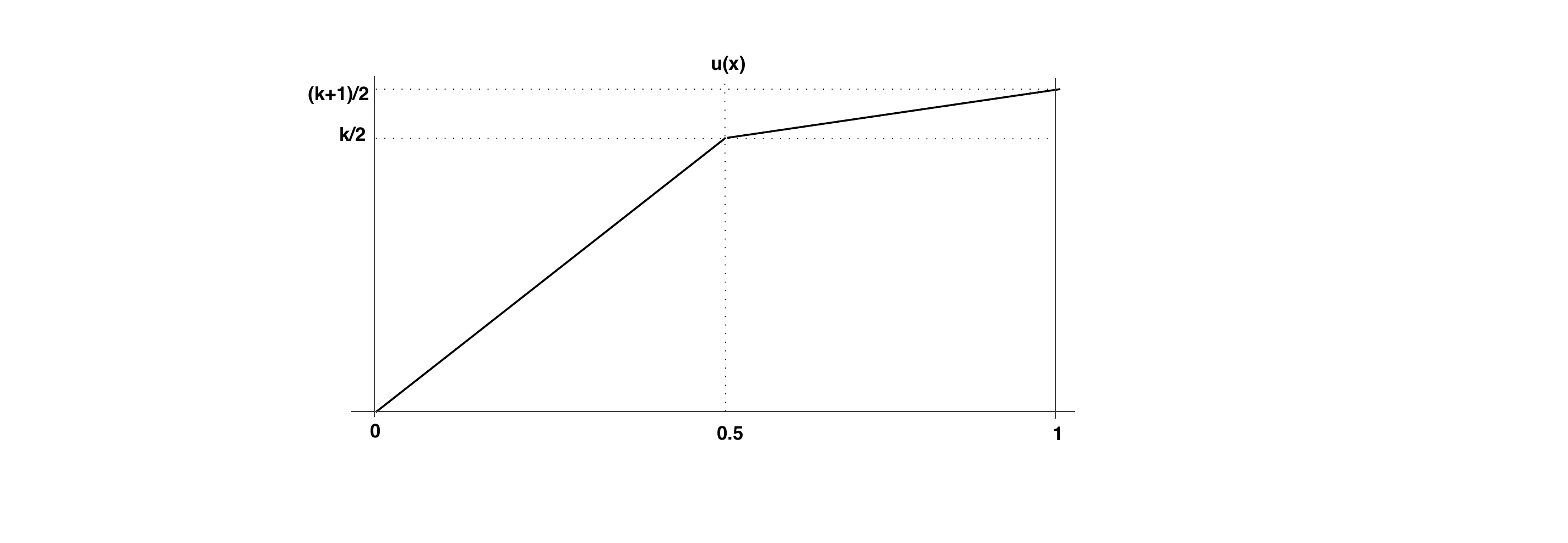}
        \caption{solution $u=u_{_\cT}$}%
        \end{minipage}%
        \quad
        \begin{minipage}[!htbp]{0.48\textwidth}
        \centering
       \includegraphics[width=0.90\textwidth,angle=0]{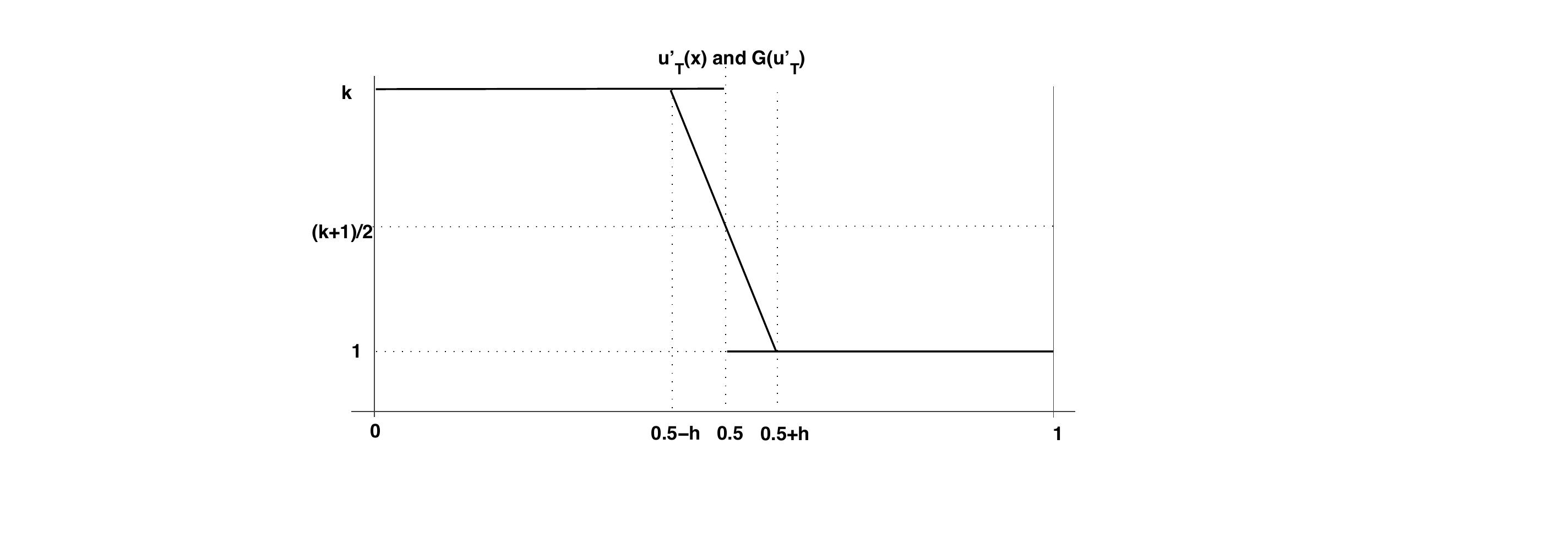}
        \caption{$u_{_\cT}^\prime(x)$ and $G(u_{_\cT}^\prime)$}%
    \end{minipage}%
        \hfill
\end{figure}

For any triangulation $\cT$ with $x=1/2$ as one of its vertices, the conforming linear finite element approximation 
is identical to the exact solution: $u_{_\cT} = u$, and hence the true error equals to zero. Without loss of generality,
assume that the size of two interface elements is $h$. Then the recovered gradient
is depicted in Figure 2 and its value at $x=1/2$ is $(k+1)/2$. A simple calculation yields the ZZ error indicator:
 \[
    \xi_{_{\ZZ,K}} = \|G(u_{_\cT}^\prime)-u_{_\cT}^\prime\|_{0,K} = \left\{\begin{array}{lll}
    \dfrac{1}{2\sqrt{3}}(k-1)\,h^{1/2},\quad &x\in (1/2-h,\, 1/2),\\[4mm]
    \dfrac{1}{2\sqrt{3}}(k-1)\,h^{1/2},\quad &x\in (1/2,\,1/2+h),\\[4mm]
    0,&\mbox{otherwise}.
    \end{array}
    \right.
 \]
Hence, no matter how small the mesh size $h$ is, 
the ZZ indicators at two interface elements could be arbitrarily large.

For this simple one-dimensional example, to overcome the inefficiency of the estimator, 
one may use the ZZ estimator based on the flux. However, this idea could not be extended to two or three dimensions.
To see this, consider the second example defined on the domain $\O=(-1,\,1)^2$ 
with scalar piecewise constant diffusion coefficient 
 \[
 A=kI \quad\mbox{for }\, y>0
 \quad\mbox{and}\quad A=I \quad\mbox{for }\,y<0,
 \]
where $k>1$ is an arbitrary constant. Choose proper Dirichlet  boundary data such that the 
exact solution of (\ref{pde}) is piecewise linear function given by
\[
 u=\left\{\!\!\begin{array}{llll} x+y &\,\, \mbox{if}&y>0,\\[2mm]
 x+k\,y& \,\,\mbox{if}& y<0.
\end{array} \right.
\]
The conforming linear finite element approximation on any triangulation 
aligned with the interface $y=0$ is identical to the exact solution, and hence the true error vanishes.
Since the exact gradient and flux 
\[
\nabla u =\left\{
\begin{array}{ll} (1,\, 1)^t, & \mbox{for }\, y>0,\\[2mm]
(1,\,k)^t, & \mbox{for }\, y< 0
\end{array}\right.
\quad\mbox{and}\quad
 \bsigma = -A \nabla u =
 \left\{
 \begin{array}{ll} (k,\, k)^t,\! & \mbox{for }\,y>0,\\[2mm]
(1,\,k)^t,\! & \mbox{for }\, y< 0
\end{array} \right.
\]
are not continuously across the interface, similar calculation to the first example yields that the 
ZZ error indicators on the interface elements based on continuous gradient or flux recovery
could be arbitrarily large 
no matter how small the mesh sizes of the interface elements are. 

\section{Improved ZZ Estimators}\label{flux}
\setcounter{equation}{0}

The second example of the previous section shows that the gradient and the flux of the exact solution
of (\ref{pde}) are not continuously across interfaces. 
This means that inefficiency of the ZZ estimator is 
caused by using continuous functions (the recovered gradient or flux) 
to approximate discontinuous functions (the true gradient or flux). In this section, we introduce improved ZZ
estimator that is efficient.

To this end, let
\[
 H(\divvr;\O)=\{\btau\in L^2(\O)^d \,:\,
 \gradt\btau\in L^2(\O)\} \not\subset H^1(\Omega)^d
 \]
with the norm $\|\btau\|_{H(\divvr;\,\O)}=\left(\|\btau\|^2_{0,\O}+\|\gradt\btau\|^2_{0,\O}
 \right)^{1/2}$ and let
 \[
  H_{g,_N}(\divvr;\O)\!
  =\{\btau\in H(\divvr;\O) \,:\,\btau\cdot \bn|_{\Gamma_\sN}= g_\sN\}.
  \]
 Denote the $\Hdiv$ conforming Raviart-Thomas ($\RT$) and 
Brezzi-Douglas-Marini
($\BDM$) spaces \cite{BoBrFo:13} of the lowest order by
 \begin{eqnarray*}
 && \RT=\{\btau\in H(\divvr;\,\Omega) \,:\,
 \btau|_K\in \RT(K)\,\,\,\,\forall\,K\in\cT\}\\[2mm]
 \mbox{and} &&
 \BDM=\{\btau\in H(\divvr;\,\Omega) \,:\,
 \btau|_K\in \BDM(K)\,\,\,\,\forall\,\,K\in\cT\},
 \end{eqnarray*}
respectively, where $RT(K)=P_0(K)^d +(x_1,\, ...,\,x_d)^t\,P_0(K)$ and $\BDM(K)=P_1(K)^d$. 
Let 
\[
\RT_{g,N}=\RT \cap H_{g,N}(\divvr;\O)
\quad\mbox{and}\quad
\BDM_{g,N}= \BDM \cap H_{g,N}(\divvr;\O).
\]

Let $u\in H^1(\Omega)$ be the exact solution of (\ref{pde}), 
it is well known that the tangential components of the gradient and the normal component of the flux
are continuous. Mathematically, we have
 \[
 ``u\in H^1(\Omega) \Longrightarrow \nabla u \in H(\mbox{curl}; \O)"
  \quad\mbox{and} \quad 
  \bsigma=-A\nabla u \in H(\divvr;\O)  ,
 \]
 where $H(\mbox{curl}; \O)\not\subset H^1(\Omega)^d$ is
 the collection of vector-valued functions that are square
integrable and whose curl are also square integrable. 
This suggests that proper finite element spaces for
recovering the gradient and the flux are the respective $H(\mbox{curl}; \O)$  and $H(\divvr;\O)$
conforming finite element spaces. 

For the conforming finite element approximation, 
the numerical gradient $\tilde{\brho}_{_\cT}$ is already in $H(\mbox{curl}; \O)$ and, hence, 
the resulting improved ZZ estimator based on the gradient recovery is identical to zero.
Since the numerical flux $\tilde{\bsigma}_{_{\cal T}}$ is not in $H_{g,N}(\divvr;\O)$, the improved ZZ estimators
introduced in \cite{CaZh:09} are based on either explicit or implicit flux recoveries in 
$\RT_{g,N}$ and $\BDM_{g,N}$.
The explicit recovery is limited to the scalar diffusion coefficient and the $\RT$ element. 
The implicit recovery requires to solve the following global $L^2$ minimization problem:
find $\bar{\bsigma}_{_\cT} \in \cV$ such that
\beq \label{global-pro}
\|A^{-1/2}(\bar{\bsigma}_{_\cT} -\tilde{\bsigma}_{_\cT} )\|_{0,\O}
  =\min_{\btau \in \cV}\,\, \|A^{-1/2}(\btau -\tilde{\bsigma}_{_\cT})\|_{0,\O},
\eeq
where $\cV = \RT_{g,N}$ or $\BDM_{g,N}$. With the recovered flux $\bar{\bsigma}_{_\cT} \in \cV$,
the improved ZZ estimator introduced in \cite{CaZh:09} is given by
 \beq\label{estimator_im}
 \xi\, (\bar{\bsigma}_{_\cT}) = \|A^{-1/2}(\bar{\bsigma}_{_\cT} -\tilde{\bsigma}_{_\cT} )\|_{0,\O}.
 \eeq

Even though the solution of (\ref{global-pro}) may be computed efficiently by a simple iterative solver,
in the remainder of this section, we derive two new explicit and efficient flux recoveries
applicable to the problem with the full diffusion tensor
based on the respective $\RT$ and $\BDM$ elements of the lowest order.

Here, we introduce some notations. Denote the set of all edges (faces) of the triangulation by $
 \cE := \cE_{_I}\cup\cE_{_D}\cup\cE_{_N},
$ where $\cE_{_I}$ is the set of all interior element edges (faces), and
$\cE_{_D}$ and $\cE_{_N}$ are the sets of all boundary edges (faces)
belonging to the respective $\Gamma_{_D}$ and $\Gamma_{_N}$. 
For each $F\in \cE$, denote by $\bn_\sF$ the unit vector normal to $F$.
For each $F\in \cE_{_I}$, let 
$K_\sF^-$ and $K_\sF^+$ be two elements sharing the common edge (face) $F$
such that the unit outward normal vector of $K_\sF^-$ coincides with
$\bn_\sF$; and let $\bx_F^{-}$ and $\bx_F^{+}$ be the vertex of $K_F^{+}$ and $K_F^{-}$ opposite to $F$, respectively.
When $F\in \cE_\sD \cup \cE_\sN$, $\bn_\sF $ is the unit
outward vector normal to $\p\O$ and denote by $K_\sF^-$ the element
having the edge (face) $F$ and by $\bx_F^-$ the vertex in $K_F^-$ opposite to $F$.

Let $\delta$ be the Kronecker delta function.
For each $F\in \cE$,  denote by $\bphi_\sF$ the global nodal basis function of $\RT$
associated with $F$, i.e., 
\beq\label{rt-basis}
   \int_{F'}\left(\bphi_{_{F}} \cdot \bn_{\sF'}\right)\, ds = \delta_{_{FF'}},
   \quad \forall\,\, F^\prime \in \cE;
\eeq
denote by $\bpsi_{_{1,F}}, \, ... ,\,  \bpsi_{_{d,F}}$ global basis functions of $\BDM$ satisfying
 \beq\label{rt_bdm_identity}
 \bpsi_{_{1,F}} + \cdots + \bpsi_{_{d,F}}=\bphi_\sF, \quad \forall\,\, F \in \cE;
   \eeq
and let 
 \[
 \RT_\sF=\mbox{span} \{\bphi_\sF\}
 \quad\mbox{and}\quad
   \BDM_\sF=\mbox{span} \{\bpsi_{_{1,F}},\, ... ,\,  \bpsi_{_{d,F}}\}.
   \]
Since $\bg_\sN$ is piecewise constant,
for any $\btau\in \RT_{g,N}$ or $\BDM_{g,N}$, 
we have that 
 \beq\label{RT_BDM_representation}
 \btau = \sum_{F\in\,  \cE_{_I}\cup\,\cE_{_D}}  \btau_\sF 
  +\sum_{F\in \, \cE_{_N}}  g_{_{N,F}}|F| \,\bphi_\sF
  \quad \mbox{ with }\,
  \btau_\sF \in  \RT_\sF \,\mbox{ or }\,  \BDM_\sF,
 \eeq
where 
$g_{_{N,F}}= g_{_N}|_\sF$ and $|F| =\mbox{meas}_{d-1}\, (F)$.

For any $K\in\cT$, restriction of the numerical flux $\tilde{\bsigma}_{_\cT}$
 on $K$ is a constant vector and has the following 
representation in $\RT(K)$ (see Lemma~4.4 of \cite{CaZh:09}):
\[
 	\tilde{\bsigma}_{_\cT}|_\sK 
	= \sum_{F\subset \partial K} \tilde{\sigma}_{_{F,K}} |F| \bphi_\sF,
\]
where $\tilde{\sigma}_{_{F,K}}=\left(\tilde{\bsigma}_{_\cT}|_K\cdot \bn_\sF\right)_\sF$ is the normal 
component of $\tilde{\bsigma}_{_\cT}|_K$ on $F$.
On each interior edge (face) $F\in \cE_\sI$, the normal component of the numerical flux has
two values
 \[
  \tilde{\sigma}_{\sF}^- = \tilde{\sigma}_{_{F,K^-_\sF}}
  \quad\mbox{and}\quad
  \tilde{\sigma}_{\sF}^+ = \tilde{\sigma}_{_{F,K^+_\sF}}.
  \]
Denote by $\bphi^{-}_{\sF}$ and $\bphi^{+}_{\sF}$ the restriction of $\bphi_{\sF}$ on
$K_\sF^-$ and $K_\sF^+$, respectively. 
Then the numerical flux 
also has the following edge (face) representation:
\beq \label{numer_flux-representation}
  \tilde{\bsigma}_{_\cT}=
 \sum_{F \in \cE} \tilde{\bsigma}_{\sF}
 \quad\mbox{with}\quad 
 \tilde{\bsigma}_{\sF}= \left\{\begin{array}{ll}
   \tilde{\sigma}_{\sF}^- |F|\bphi^{-}_{\sF} 
   	+ \tilde{\sigma}_{\sF}^+ |F|\bphi^{+}_{\sF},
     & \forall\,\, F \in \cE_\sI,  \\[2mm]
     \tilde{\sigma}_{\sF}^- |F|\bphi^{-}_{\sF} ,
      & \forall\,\, F \in \cE_\sD \cup \cE_\sN.
  \end{array}\right.
\eeq
For any $\btau\in \RT_{g,N}$ or $\BDM_{g,N}$, (\ref{RT_BDM_representation}) and 
(\ref{numer_flux-representation}) give 
 \[
 \btau -  \tilde{\bsigma}_{_\cT}
 	= \sum_{F \in \cE_\sI } \left(\btau_\sF - \tilde{\bsigma}_{\sF} \right)
  	+ \sum_{F \in \cE_\sD} \left(\btau_\sF - \tilde{\sigma}_{\sF}^- |F|\, \bphi_\sF^-\right)
 	+ \sum_{F \in \cE_\sN} \left(g_{_{N,F}} - \tilde{\sigma}_{\sF}^-\right) \, 
   |F|\bphi_\sF^-,
\]
which, together with the triangle inequality and the choice of 
$\btau_\sF =\tilde{\sigma}_{\sF}^- \, |F|\bphi_\sF^-$ for all $F\in\cE_\sD$, implies
 \begin{eqnarray} \nonumber
  && \xi\, (\bar{\bsigma}_{_\cT})
    = \min_{\btau \in \cV} \,\,\|A^{-1/2} \left( \btau -  \tilde{\bsigma}_{_\cT} \right)  \|_{0,\Omega}
   \\[2mm]  \label{global-local1} 
 &\leq & \sum_{F\in \cE_I} \min_{\btau_{\sF} \in \cV_{\sF}} \,\,
         \|A^{-1/2} \left( \btau_{\sF}- \tilde{\bsigma}_{\sF} \right)  \|_{0,\o_{\sF}}  
       +\sum_{F\in \cE_N} \|A^{-1/2} 
       \left( g_{_{N,F}} - \tilde{\sigma}_{\sF}^-\right) \, |F| \bphi_\sF^-\|_{0,K_F^-},
\end{eqnarray}
where $\o_\sF$ is the union of elements sharing the edge (face) $F$ for all $F \in \cE$,
$\cV = \RT_{g,N}$ or $\BDM_{g,N}$ and $\cV_\sF=\RT_\sF$ or $ \BDM_\sF$.

For each $F\in \cE_{_I}$, let $\hat{\bsigma}_\sF \in \cV_F =\RT_\sF$ or $ \BDM_\sF$ be the solution of the following 
local minimization problem:
 \beq\label{local_minimization}
 \|A^{-1/2}\left(\hat{\bsigma}_\sF - \tilde{\bsigma}_{\sF} \right)  \|_{0,\o_{\sF}} 
  =\min_{\btau \in \cV_\sF}\,\, \|A^{-1/2} \left(\btau - \tilde{\bsigma}_{\sF} \right)  \|_{0,\o_{\sF}} ,
 \eeq
by (\ref{global-local1}),  it is then natural to introduce the following edge (face) based estimator and indicators
\beq\label{improvedZZ:F}
 \hat{\xi} = \sum_{F\in \cE_I\, \cup\, \cE_N} \xi_\sF
 \quad\mbox{with}\quad 
 \xi_\sF= \left\{\begin{array}{ll}
    \|A^{-1/2}\left(\hat{\bsigma}_\sF - \tilde{\bsigma}_{\sF} \right)  \|_{0,\o_{\sF}}  & F\in \cE_I, \\[2mm]
     \|A^{-1/2} \left(  g_{_{N,F}} 
     - \tilde{\sigma}_{\sF}^-\right) \, |F|\bphi_\sF^-\|_{0,K_F^-} & F\in \cE_N,
     \end{array}\right.
\eeq
which satisfies
  \beq\label{eff_xihat}
   \xi\, (\bar{\bsigma}_{_\cT}) \leq \hat{\xi}.
\eeq
To introduce the element based estimator, define the recovered flux 
$\hat{\bsigma}_{_\cT}\in \RT_{g,N}$ or $\BDM_{g,N}$ as follows:
 \beq\label{recovered_flux}
  \hat{\bsigma}_{_\cT} = \sum_{F\in  \cE_{_I}}  \hat{\bsigma}_\sF 
   + \sum_{F\in  \cE_{_D}}  \tilde{\sigma}_\sF^- \,|F|\bphi_\sF
  +\sum_{F\in  \cE_{_N}}  g_{_{N,F}} \, |F|\bphi_\sF.
  \eeq
  Then the element based indicators and estimator are given by 
\beq\label{improvedZZ}
 	\xi_\sK=\|A^{-1/2}(\hat{\bsigma}_{_\cT}-\tilde{\bsigma}_{_\cT})\|_{0, K}, \,\,
 	\;\forall\,\, K\in\cT
 	\quad\mbox{and}\quad
 	\xi=\|A^{-1/2}(\hat{\bsigma}_{_\cT}-\tilde{\bsigma}_{_\cT})\|_{0,\Omega}.
\eeq

The minimization problem in (\ref{local_minimization}) is equivalent to the following variational problem: find 
$\hat{\bsigma}_\sF \in \cV_{\sF}$ such that
 \beq \label{local-pro}
  \left(A^{-1}\, \hat{\bsigma}_\sF,\, \btau \right)_{\o_{\sF}}
      =\left(A^{-1} \,\tilde{\bsigma}_{\sF},\, \btau \right)_{\o_{\sF}} 
      \quad \forall\,\, \btau\in\cV_\sF.
\eeq
The local problem in (\ref{local-pro}) has only one unknown if $\cV_{\sF}=\RT_\sF$ and 
$d$ unknowns if $\cV_{\sF}=\BDM_\sF$. The explicit formula of the solution $\hat{\bsigma}_\sF$
will be given in section~6.

\section{Efficiency and Reliability}\label{estimators-b}
\setcounter{equation}{0}

This section establishes efficiency and reliability bounds of the indicators and estimators defined in 
(\ref{improvedZZ:F}) and (\ref{improvedZZ}),
respectively, for the diffusion problem with the coefficient matrix $A$ being locally similar to the identity matrix. 

To this end, for each $K \in \cT$, denote by $\lambda_{\max,K}$ and $\lambda_{\min,K}$
the maximal and minimal eigenvalues of $A_K=A\big|_K$, respectively. Let 
 \[
  \lambda_{\max} =\max \limits_{K\, \in\, \cT}\, \lambda_{\max,K} 
  \quad\mbox{and}\quad
  \lambda_{\min} =\min \limits_{K\, \in\, \cT}\, \lambda_{\min,K} .
  \]
Assume that each local matrix $A_K$ is similar to the identity matrix in the sense that its maximal and minimal eigenvalues are 
almost of the same size, i.e., 
there exists a moderate size constant $\kappa>0$ such that
\beq\label{kappa}
	\dfrac{\lambda_{\max,K}}{\lambda_{\min,K}}\, \le\,  \kappa,  \quad \forall \;\,K\, \in \cT.
\eeq
Nevertheless, the ratio of the global maximal and minimal eigenvalues, $\lambda_{\max} \big/\lambda_{\min}$,
could be very large.
In order to show that the reliability constant is independent of the ratio,
 we assume that the distribution of 
$\lambda_{\min}(x)$ is quasi-monotone (see \cite{Pet:02}). 

Let $\Gamma_\sI$ be the set of all interfaces of the diffusion coefficient that are assumed to be 
aligned with element interfaces, 
and denote by $f_z= \dfrac{1}{ \mbox{meas}_d(\o_z)} \int_{\o_z} f \,dx$ the average of $f$ over $\omega_z$.
Let 
\[
	\widehat H_f=\left\{ 	
		\sum_{z \in \cN \setminus ( \overline{\Gamma}_\sI \cup \overline{\Gamma}_D) } 
			\dfrac{\mbox{meas}_{d} (\o_z)}{\lambda_{\min,\o_z}} \, \|f-f_z\|_{0,\o_z}^2
		+\sum_{z \in \cN \cap ( \overline{\Gamma}_\sI \cup \overline{\Gamma}_D)} \sum_{K \subset \o_z}
		\dfrac{h_K^2}{\lambda_{\min,K}} \, \|f\|_{0,K}^2		
	\right\}^{1/2}.
\]
 Note that $A$ is a constant matrix in $\o_z$ if $z \in \cN \setminus ( \overline{\Gamma}_\sI \cup \overline{\Gamma}_D)$.

\begin{rem}  For various lower order finite element approximations,
the first term in $\widehat{H}_f$ is of higher order
than $\eta_{_F}$ {\em (}defined below in {\em (\ref{edgeestimator-c}))}
  for $f\in L^2(\O)$ and so is the second term for $f\in L^p(\O)$
with $p > 2$ {\em (}see {\em \cite{CaVe:99}}{\em )}.
\end{rem}

\begin{thm} {\em (Global Reliability)}
Assume that the distribution of $\lambda_{min,K}$ is quasi-monotone. Then
the error estimators $\hat \xi$ and $\xi$ defined in {\em (\ref{improvedZZ:F})} and {\em (\ref{improvedZZ})}  , respectively, satisfies the 
following global reliability bound:
\beq\label{rel-rt}
\|A^{1/2} \nabla(u-u_{_\cT})\|_{0,\O} \leq C\, \left(\xi + \widehat H_f \right),
\eeq
and
\beq\label{rel-F}
\|A^{1/2} \nabla(u-u_{_\cT})\|_{0,\O} \leq C\, \left(\hat \xi + \widehat H_f \right),
\eeq
where the constant $C$ depends on the shape regularity of $\cT$ and $\kappa$, 
but not on $\lambda_{\max} \big/ \lambda_{\min}$.
\end{thm}

\begin{proof} 
It follows from (\ref{numer_flux-representation}), (\ref{recovered_flux}), and Young's inequality that
\begin{eqnarray*}
	&&\xi^2=\left\|\sum_{F \in \cE_I \cup \cE_N} A^{-1/2}(\hat \bsigma_\sF -\tilde \bsigma_\sF)\right\|^2_{0,\,\O}\\
	&=&\!\!\!\!\!\!
	\sum_{ F \in \cE_I \cup \cE_N} \!\!\!
	\left(
		\!\sum_{ F' \subset \partial K_\sF^-}
		\!\! \left(A^{-1}(\hat \bsigma_\sF -\tilde \bsigma_\sF),(\hat \bsigma_{_{F'}} -\tilde \bsigma_{_{F'}})\right)_{K_F^-}
		+\!\!\!\!\! \!\!\sum_{ F' \subset \partial K_\sF^+} \!\!
		\left(A^{-1}(\hat \bsigma_\sF -\tilde \bsigma_\sF),(\hat \bsigma_{_{F'}} -\tilde \bsigma_{_{F'}})\right)_{K_F^+} 		\right)
		\\[2mm]
		&\le& \sum_{ F \in \cE_I\, \cup\, \cE_N} \left(\dfrac{d}{2}+1\right) \|A^{1/2} (\hat \bsigma_\sF -\tilde \bsigma_\sF)\|^2_{0,\,\o_F}
	 = \left(\dfrac{d}{2}+1\right) \hat \xi^2.
\end{eqnarray*}
Now, it suffices to prove the validity of (\ref{rel-rt}).
Note that for any $K\in\cT$ and for any vector filed $\btau$, we have that
\[
	\lambda_{\min,K}^{1/2} \, \| \btau\|_{0,K}
	\le \|A^{1/2}\, \btau \|_{0,K}
	 \le \lambda_{\max,K}^{1/2} \, \|\btau\|_{0,K},
\]
and that
\beq \label{A inverse bound}
	\lambda_{\max,K}^{-1/2} \, \| \btau\|_{0,K}
	\le \|A^{-1/2}\, \btau \|_{0,K}
	 \le \lambda_{\min,K}^{-1/2}\, \|\btau\|_{0,K}.
\eeq
With the above inequalities, (\ref{rel-rt}) may be proved in a similar fashion as 
in \cite{CaZh:09, CaZh:10c} with the constant $C$ also depending on $\kappa$.
\end{proof}

In the remaining part of this section, we will establish the efficiency of the indicators
$\xi_\sF$ and $\xi_\sK$
given in (\ref{improvedZZ:F}) and (\ref{improvedZZ}), respectively,
by proving that the indicators $\xi_\sF$ and $\xi_\sK$ are bounded above
by the classical residual based indicators of the flux jump on edges (faces) given 
in (\ref{edgeestimator-c}),
which is well known to be efficient 
for interface problems, (i.e., 
$A=\a(x)\, I$ with $\a(x)$ being a piecewise constant function). More specifically,
Petzoldt (see (5.7) in \cite{Pet:02}) proved that the edge (face) flux indicator

\beq\label{edgeestimator-c}
 \eta_{\sF} = \left\{\begin{array}{lll}
|F|\, |\tilde \sigma_\sF^- - \tilde \sigma_\sF^+|\big/ \sqrt{\alpha_\sF^+ + \alpha_\sF^-}, & F\in\cE_\sI,\\[4mm]
  |F|\, |\tilde \sigma_\sF^- -g_N|\big/ \sqrt{ \alpha_\sF^-},  & F\in \cE_\sN,\\[4mm]
  0,&F \in \cE_D,
\end{array}
\right.
\eeq
where $\a_F^\pm=\a_{K_F^{\pm}}$, is locally efficient without 
assumption on the distribution of the coefficient $\alpha$. More specifically, 
there exists a constant $C>0$ independent of $\alpha$ and the mesh size such that
\beq \label{edgeestimator}
\eta_{\sF}^2 \leq C\left( \|\alpha^{-1/2}  \nabla(u-u_{_\cT}) \|_{0,\o_\sF}^2 
 + \sum_{K\, \subset\, \o_F} \dfrac{h_K^2}{\alpha_\sK}\, \|f-f_{_\cT}\|_{0,K}^2
\right).
\eeq
where $f_{_\cT}$ is the $L^2$ projection of $f$ onto the space of piecewise constant with respect to $\cT$. 

\begin{rem}
For the diffusion problem, define the 
edge (face) estimator $\eta_\sF$ according to {\em (\ref{edgeestimator-c})} with 
$\a_F^{\pm}=\lambda_{\min,K_F^{\pm}}$,
then 
the local efficiency in {\em (\ref{edgeestimator})} holds
with $\a =A$ and the constant $C$  depending on $\kappa$.
\end{rem}

\begin{thm}{\em (Local Efficiency)\label{thm:eff-c}}
The local edge (face) and element indicators defined in {\em{(\ref{improvedZZ:F})}} and \em{(\ref{improvedZZ})}, respectively,
 are efficient, 
i.e., there exists a constant $C>0$ depending only on the shape regularity
of $\cT$ and $\kappa$ such that
\beq \label{edge-eff}
	\xi^{bdm}_{\sF} \le \xi^{rt}_{\sF}  \leq C\left( \|A^{1/2} \nabla e_{_{\cal T}}\|_{0,\o_\sF}+
	\left( \sum_{K' \subset \o_\sF} 
	\dfrac{h_{_{K'}}^2}{\alpha_{_{K'}}}\, \|f-f_{_\cT}\|_{0,K'}^2\right)^{1/2} \right), \quad \forall \,F\in \cE
 \eeq
and that
\beq \label{c-eff}
	\xi^{bdm}_{\sK}, \,\, \xi^{rt}_{\sK}  \leq C\left( \|A^{1/2} \nabla e_{_{\cal T}}\|_{0,\o_\sK}
	+
	\left( \sum_{K' \subset \o_\sK} \dfrac{h_{_{K'}}^2}{\alpha_{_{K'}}}\, \|f-f_{_\cT}\|_{0,K'}^2\right)^{1/2} \right), 
	\quad \forall \,K \in \cT,
 \eeq
where $\o_\sK$ is the union of all elements that shares at least one edge (face) with $K$.
\end{thm}

\begin{proof}
It follows from (\ref{numer_flux-representation}), 
(\ref{recovered_flux}), and the triangle inequality that
\[
	\xi_\sK =\| A^{-1/2} (\hat \bsigma_{_\cT} -\tilde \bsigma_{_\cT})\|_{0,K}
	            \leq  \sum_{F\subset \partial K} \| A^{-1/2} (\hat \bsigma_\sF-\tilde \bsigma_\sF)\|_{0,K}
	\leq \sum_{F\subset \partial K} \xi_{\sF}, \quad \forall \, K \in \cT.
\]
Hence, (\ref{c-eff}) is a direct consequence of (\ref{edge-eff}). 

To prove the validity of (\ref{edge-eff}), first note that the first inequality is a direct consequence of 
the minimization problem in (\ref{local_minimization}) and the fact that 
$\cV_\sF^{rt} \subset \cV_\sF^{bdm}$.
To prove the second inequality in (\ref{edge-eff}), 
without loss of generality, assume that 
$F \in \cE_I$ and that 
$\lambda_{\min,K_F^-}\ge \lambda_{\min,K_F^+}$.
By (\ref{A inverse bound}) and the fact that
$\|\bphi_F\|_{0,K_F^-} \le C$ with constant $C>0$ 
depending only on the shape regularity of $\cT$,
we have
\begin{eqnarray*} \label{effi:a}
	\xi^{rt}_\sF& =& \min_{\btau \in \RT_\sF}\|A^{-1/2} (\btau-\tilde \bsigma_{\sF})\|_{0,\o_\sF}
	\le \|A^{-1/2} (\tilde\sigma_{\sF}^+ \bphi_\sF-\tilde \bsigma_{\sF})\|_{0,\o_\sF}\\[2mm] 
	&=& \| \left(\tilde \sigma_\sF^+ - \tilde \sigma_\sF^-\right)A^{-1/2} |F|\bphi_\sF\|_{0,K_\sF^-}
	\le   \left |\tilde \sigma_\sF^+ - \tilde \sigma_\sF^-\right|\, \lambda_{\min,K_F^-}^{-1/2}\,
	        |F|\large\|  \bphi_\sF \large\|_{0,K_\sF^-}\\[2mm]
	&\le& C\, |F|\, \left |\tilde \sigma_\sF^+ - \tilde \sigma_\sF^-\right|  \, 
		\left(\lambda_{\min,K_F^-}+\lambda_{\min,K_F^+} \right)^{-1/2}.
\end{eqnarray*}
Combining with Remark 5.3 
implies the second inequality in (\ref{edge-eff}) 
and, hence, 
(\ref{c-eff}). This completes the proof of the theorem.
\end{proof}

\section{Explicit Formulas}
\setcounter{equation}{0}

This section presents explicit formulas of
the recovered fluxes defined in (\ref{recovered_flux}) (see (\ref{rt-explicit1}) and (\ref{bdm-explicit1}))
and the corresponding indicators and estimators defined in (\ref{improvedZZ:F}) and (\ref{improvedZZ}), respectively. In particular, the explicit formulas
for the indicators and, hence, the estimators are written in terms of the current approximation $u_{_\cT}$ 
and geometrical information
of elements. For simplicity, we only consider the two-dimensional case.

For each edge $F \in \cE$, denote  
by $\bs_\sF$  and $\be_\sF$ the globally fixed initial and terminal 
points of $F$, respectively, such that $\bs_\sF-\be_\sF = |F|\, \bt_\sF$ with 
$\bt_\sF =(t_{_{1,F}}, t_{_{2, F}})^t$ being a unit vector tangent to $F$;
by $\bn_\sF= (t_{_{2, F}}, -t_{_{1, F}})$ a unit vector normal to $F$; and by
 $\bx_\sF^{\pm}$ the opposite vertices of $F$ in $K_\sF^{\pm}$, respectively.
 
Denote by $\lambda_{\bs_\sF}$ and 
$\lambda_{\be_\sF}$ the nodal basis functions of the continuous linear element associated with vertices
$\bs_\sF$ and $\be_\sF$ of $\cN$, respectively.
For any $v\in H^1(\Omega)$, denote the formal adjoint of the curl operator by 
$\gperp v = (\partial v/\partial y, -\partial v/\partial x)^t$.
For the $\RT$ space of the lowest index, the nodal basis function associated with $F \in \cE$ is given by
$$
	\bphi_\sF = \left( \lambda_{\bs_\sF} \gperp \lambda_{\be_\sF} 
		- \lambda_{\be_\sF} \gperp \lambda_{\bs_\sF} \right).
$$
For the $\BDM$ space of the lowest index, two basis functions associated with the edge $F\in\cE$ are given by 
$$
	\bpsi_{s,\sF} = \lambda_{\bs_\sF} \gperp \lambda_{\be_\sF} 
	\quad \mbox{and}\quad
	\bpsi_{e,\sF}= -  \lambda_{\be_\sF} \gperp \lambda_{\bs_\sF},
$$ 
respectively, which satisfy
 \[
 \left(\bpsi_{_{\bs,\sF}} \cdot \bn_{\sF'}\right)\Big|_{{F'}} = \lambda_{\bs_\sF} \delta_{_{FF'}}/|F'|
   \quad\mbox{and}\quad
     \left(\bpsi_{_{\be,\sF}} \cdot \bn_{\sF'}\right)\Big|_{{F'}} =  \lambda_{\be_\sF}\delta_{_{FF'}}/|F'|
 \]
 for any $F^\prime \in \cE$. It is easy to check that (\ref{rt-basis}) 
 and (\ref{rt_bdm_identity}) hold.
 
A detaied MATLAB implementation of $\BDM$ and $RT$ mixed finite element methods can be found in \cite{Zhang:15}.
 
\subsection{Indicator and Estimator Based on $\RT$}

For all $F\in \cE_\sI$, let
 \[
 \gamma_{\sF}^\pm = \left(A^{-1} \bphi_{\sF},\,\bphi_{\sF}\right)_{K_\sF^\pm} 
 \quad\mbox{and}\quad
 a_{\sF}
 = \dfrac{ \gamma_{\sF}^- } { \gamma_{\sF}^- +\gamma_{\sF}^+  }.
 \]
 Using the basis function $\bphi_\sF$ defined above, a straightforward calculation gives

  \[
  	\gamma_{F}^{\pm}=
 		 \dfrac{1}{48|\, K_F^{\pm}|} 
		 	\left( \sum_{F' \subset \partial K_F^{\pm}} 
			\left\| A^{-1/2}(\bx_{_{F'}}^{\pm}-\bx_\sF^{\pm}) \right\|^2
  			+\left\|A^{-1/2} \left(\sum_{F' \subset \partial K_F^{\pm}}
			 \bx_{_{F'}}^{\pm}-3\bx_\sF^{\pm} \right) \right\|^2
		\right),
  \]
where $\| \cdot \|$ is the standard Euclidean norm in $\cR^2$. 
Solving the local problem in (\ref{local-pro}) with $\cV_\sF=\RT_\sF$ gives the following 
recovered flux in $\RT_{g,N}$:
 \beq\label{rt-explicit1}
  \hat \bsigma_{_\cT}^{rt}
 = \sum_{F\in\cE_\sI}\hat \sigma_{\sF}^{rt} \, |F|\bphi_\sF
   + \sum_{F\in\cE_\sD}\tilde{\sigma}^-_{\sF}\, |F|\bphi_\sF^- 
   + \sum_{F\in\cE_\sN} g_{_{N,F}}\, |F|\bphi_\sF^-
\eeq
with the normal component of the recovered flux, $ \hat\sigma_{_{F}}^{rt}$, on each edge  $F \in \cE_I$ given
 by the following weighted average:
 \beq\label{rt-coef}
 \hat\sigma_{_{F}}^{rt} =  a_{\sF}\,  \tilde{\sigma}^-_{\sF}  +\left(1-a_{\sF}\right) \, \tilde{\sigma}^+_{\sF}.
  \eeq
The edge indicator $\xi_\sF^{rt}$ has the following explicit formula:
 \[
 	\xi_\sF^{rt}=\left\{
	\begin{array}{lll}
		|\tilde \sigma_\sF^- - \tilde \sigma_\sF^+| \,  |F| 
			\left( (1-a_\sF)^2\gamma_F^- + a_\sF^2 \gamma_\sF^+ \right)^{1/2},& F\in \cE_I,\\[2mm]
		0,&F\in \cE_D,\\[2mm]
		|\tilde \sigma_\sF^- -g_{_{N,F}}| \, |F|\,  \sqrt{\gamma_\sF^-},&F\in \cE_N.
	\end{array}
 	\right.
 \]
 
Next, we introduce explicit formula of the element 
indicator $\xi_\sK^{rt}$ in terms of the current approximation $u_{_\cT}$ and geometrical information of elements.
To this end, for any $K \in \cT$, denote the sign function by 
\[
\mbox{sign}_\sK(F)=\left \{\begin{array}{lll}
1 & \mbox{if} & \bn_\sF=\bn_\sK|_\sF,\\[4mm]
-1 & \mbox{if} & \bn_\sF=-\bn_\sK|_\sF,
\end{array}
\right.
\quad \forall \,F \subset \partial K,
\]
where $\bn_\sK$ is the unit outward vector normal to $K$, and let 
\begin{eqnarray*}
	R_{_{FF'}}\!\!\!\!&\!\!\!=\!\! \!\!& \!\!\!\!
	 \left(\!\sum_{F'' \subset \partial K} \!\! \bx_{_{F''}}-3\bx_{_{F'}} \!\! \right)^t \!\!
	 A_\sK^{-1}  \! \left(\!\sum_{F'' \subset \partial K}  \!\! \bx_{_{F''}}-3\bx_{_{F}}  \!\!\right)\!\!
		+\! \! \!\!\!\ \sum \limits_{F'' \subset \partial K} \!\!\!
		(\bx_{_{F''}}-\bx_{_{F'}})^t A_\sK^{-1}(\bx_{_{F''}}-\bx_{\sF}) , \\[2mm]
  	T_{_{FF'}} &=&
		\bn_{_{F'}}^t\left( \sum_{F'' \subset \partial K}^3 \bx_{_{F''}}-3\bx_\sF\right),
		\quad \mbox{and} \quad
	S_{_{FF'}}	= \bn_{_{F'}}^t A_\sK \bn_{_{F}}.
\end{eqnarray*}
For each element $K\in\cT$, the indicator $\xi^{rt}_\sK$ is given by
 \beq \label{rtK:1}
 \xi_\sK^{rt} =\left( \left(A^{-1}\hat\bsigma_{_\cT}^{rt}, \, \hat\bsigma_{_\cT}^{rt} \right)_\sK
	+2\,\left(\hat\bsigma_{_\cT}^{rt}, \,  \nabla u_{_{\cT}} \right)_\sK
	+ \left(A\nabla u_{_{\cT}}, \, \nabla u_{_{\cT}} \right)_\sK \right)^{1/2}
\eeq 
with explicit formula for each term as follows 
\begin{eqnarray*}   
	\left(A^{-1}\hat\bsigma_{_\cT}^{rt},\, \hat\bsigma_{_\cT}^{rt} \right)_\sK
	\!\!\!\!&=&\!\!\!\!
	\dfrac{1 }{48\, |K|}   \sum_{F \subset \partial K} \sum_{F' \subset \partial K}
	\mbox{sign}_K(F) \, \mbox{sign}_K(F') \,|F||F'|
	\hat{\sigma}_{\sF}^{rt} \, \,  \hat{\sigma}_{_{F'}}^{rt} R_{_{FF'}}, \nonumber\\[2mm]
	\left(\hat\bsigma_{_\cT}^{rt}, \, \nabla u_{_{\cT}} \right)_\sK
	\!\!\!\! &=& \!\!\!\!-
	 \dfrac{1}{12\, |K|}  \sum_{F \subset \partial K}\sum_{F' \subset \partial K}
	\mbox{sign}_K(F)\, \mbox{sign}_K(F')\, |F||F'|\, \hat\sigma_{\sF}^{rt} \, u_{_\cT}(\bx_{_{F'}}) \,
	T_{_{FF'}},\nonumber\\[2mm]
	\mbox{and } \,\,\nonumber  
	\left(A \nabla  u_{_\cT}, \,\nabla  u_{_\cT} \right)_\sK
	\!\!\!\! &=&\!\!\!\!\!
	 \dfrac{ 1} {4\, |K|} \sum_{F \subset \partial K}    \sum_{F' \subset \partial K} 
	 u_{_\cT}(\bx_{_F})  \, u_{_\cT}(\bx_{_{F'}})\mbox{sign}_K(F)\mbox{sign}_K(F') \, |F||F'|
	S_{_{FF'}}.
\end{eqnarray*}

\begin{rem}
For interface problems, the recovered flux in {\em (\ref{rt-explicit1})}
and the resulting estimator defined in {\em (\ref{improvedZZ})}
are equivalent to those introduced  and analyzed in 
{\em \cite{CaZh:09}}.
To see that, let $A|_\sK = \a_\sK I$  for any $K\in\cT$,  where $\a_\sK$ and $I$
are constant and the identity matrix, respectively.
Let 
 \[
 \a_\sF^-=\a_{_{K^-_\sF}}
\quad{and}\quad\a_\sF^+=\a_{_{K^+_\sF}}, 
\]
then
 \[
 \gamma_\sF^- = \dfrac{1}{\a_\sF^-} \left(\bphi_\sF,\,
 \bphi_\sF\right)_{K_\sF^-}
 \quad\mbox{and}\quad
 \gamma_\sF^+ = \dfrac{1}{\a_\sF^+}
 \left(\bphi_\sF, \,\bphi_\sF\right)_{K_\sF^+}.
 \]
For a regular triangulation, the ratio of $\left(\bphi_\sF,\,
\bphi_\sF\right)_{K_\sF^-} $ and $\left(\bphi_\sF,\,
\bphi_\sF\right)_{K_\sF^+} $ is bounded above and below by constants. 
Thus
 \beq\label{c-equi}
 a_{\sF}=\dfrac{\gamma_{\sF}^-}{\gamma_{\sF}^- + \gamma_{\sF}^+}
 \approx  \dfrac{\a_\sF^+}{\a_\sF^- + \a_\sF^+}
 \quad \mbox{and}\quad
 1-a_{\sF}=\dfrac{\gamma_{\sF}^+}{\gamma_{\sF}^- + \gamma_{\sF}^+} \approx
 \dfrac{\a_\sF^-}{\a_\sF^- + \a_\sF^+}.
 \eeq
{\em (}Here, we use $x \approx y$ to mean that there
exist two positive constants $C_1$ and $C_2$ independent of the mesh
size such that $C_1 x \leq y\leq C_2 x$.{\em )}
{\em (\ref{c-equi})} indicates
that the weights in {\em (\ref{rt-explicit1})} may be
replaced by the respective $\dfrac{\a_\sF^+}{\a_\sF^- + \a_\sF^+}$ and
$\dfrac{\a_\sF^-}{\a_\sF^- + \a_\sF^+}$ 
Hence, it is equivalent to the explicit estimator introduced in {\em \cite{CaZh:09}}. 
\end{rem}

\subsection{Indicator and Estimator Based on $\BDM$}

 For all $F\in \cE_\sI$ and for $i,\, j \in\{ s,\,e\}$, let
  \[
  \beta_{ij,\sF}^\pm =   \left(A^{-1} \bpsi_{i,\sF},\,  \bpsi_{j,\sF}\right)_{K_\sF^\pm}
  \quad\mbox{and}\quad \beta_{ij,\sF}  = \beta_{ij,\sF}^- +\beta_{ij,\sF}^+ ,
  \]
and let
 \begin{eqnarray*}
 b_{s,\sF} &=&  \dfrac{ (\beta_{ss,\sF}^- + \beta_{se,\sF}^-) \, \beta_{ee,\sF}
          -(\beta_{se,\sF}^- +\beta_{ee,\sF}^- ) \, \beta_{se,\sF}  }
         {\beta_{ss,\sF}\beta_{ee,\sF}- \beta_{se,\sF}^2}   \\[4mm]
  \quad \mbox{and } \,\, b_{e,\sF} &=& \dfrac{ (\beta_{se,\sF}^- + \beta_{ee,\sF}^-) \, \beta_{ss,\sF}
              -(\beta_{ss,\sF}^- +\beta_{se,\sF}^- ) \, \beta_{se,\sF}  }
           {\beta_{ss,\sF}\beta_{ee,\sF}- \beta_{se,\sF}^2} .
 \end{eqnarray*}
 Using the basis functions $\bpsi_{s,\sF}$ and $\bpsi_{e,\sF}$ defined at the beginning of this section,
 a straightforward calculation gives that
 \begin{eqnarray*}
  \b_{ss,\sF}^{\pm} 
   &=& 
   	\dfrac{1}{24|K_\sF^{\pm}|} \|A^{-1/2} (\bx_\sF^{\pm} -\bs_\sF) \|^2,
   \quad  \b_{ee,\sF}^{\pm}
      =
         	\dfrac{1}{24|K_\sF^{\pm}|} \|A^{-1/2} (\bx_\sF^{\pm} -\be_\sF) \|^2 \\[4mm]
   \quad \mbox{and} &&
    \b_{se,\sF}^{\pm}= \dfrac{ (\bx_\sF^{\pm} - \bs_\sF) \, A_\sK^{-1}\,  (\bx_\sF^{\pm} -\be_\sF)  }{48\, |K_\sF^{\pm}|}.
 \end{eqnarray*}

Solving the local problems in (\ref{local-pro}) with $\cV_\sF=\BDM_\sF$ gives the following 
recovered flux in $\BDM_{g,N}$:
 \beq\label{bdm-explicit1}
    \hat\bsigma_{_{\cal T}}^{bdm} 
 	=\sum_{F \in \cE_\sI} \left(\hat\sigma_{s,\sF}^{bdm} \,  \bpsi_{s,\sF} 
	 +\hat\sigma_{e,\sF}^{bdm} \,   \bpsi_{e,\sF} \right)|F|
	 + \sum_{F\in\cE_\sD}\tilde{\sigma}^-_{\sF}\, |F|\bphi_\sF^-
	 + \sum_{F\in\cE_\sN} g_{_{N,F}}\, |F|\bphi_\sF^-
  \eeq
with the normal components of the recovered flux 
 given by the weighted averages:
  \[ 
  \hat\sigma_{s,\sF}^{bdm} =  b_{s,\sF} \,\tilde{\sigma}^-_{\sF} +(1- b_{s,\sF})\, \tilde{\sigma}^+_{\sF}
  \quad\mbox{and}\quad
   \hat\sigma_{e,\sF}^{bdm} = b_{e,\sF}\, \tilde{\sigma}^-_{\sF} +(1-b_{e,\sF}) \,\tilde{\sigma}^+_{\sF}.
   \]
 The edge indicator $\xi_\sF^{bdm}$ has the following explicit formula:
 \[
 	\xi_\sF^{bdm}=\left\{
	\begin{array}{lll}
		|\tilde \sigma_\sF^- - \tilde \sigma_\sF^+| \,  |F| \,  w_F^{1/2} & F\in \cE_I,\\[2mm]
		0,&F\in \cE_D,\\[2mm]
		|\tilde \sigma_\sF^- -g_{_{N,F}}| \, |F| 
			\left( \beta_{_{ss,F}}^-+2 \beta_{_{se,F}}^-+\beta_{_{ee,F}}^- \right)^{1/2} ,&F\in \cE_N,
	\end{array}
 	\right.
 \]
 where $w_\sF$ is given by
 \[
 	w_\sF=
			(1-b_{_{s,F}})^2 \beta_{_{ss}}^-
			+2(1-b_{_{s,F}})(1-b_{_{e,F}}) \beta_{_{se}}^-
			+(1-b_{_{e,F}})^2 \beta_{_{ee}}^-
			+b_{_{s,F}}^2 \beta_{_{ss}}^+
			+2 b_{_{s,F}} b_{_{e,F}} \beta_{_{se}}^+
			 +b_{_{e,F}}^2 \beta_{_{ee}}^+.
 \]

Next, we present explicit formula of the element indicator $\xi_\sK^{bdm}$ in terms of the current
approximation $u_{_\cT}$ and geometrical information of elements.
For each $F \subset \partial K$, denote by $F_{\bs}$ and $F_{\be}$ the remaining two edges of $K$ that is opposite to 
$\bs_\sF$ and $\be_\sF$, respectively. 
Then the indicator $\xi_\sK^{bdm}$ is computed by three terms as follows: 
 \beq \label{bdmK:1}
\xi_\sK^{bdm} =\left( \left(A^{-1}\hat\bsigma_{_\cT}^{bdm}, \, \hat\bsigma_{_\cT}^{bdm} \right)_\sK
	+2\,\left(\hat\bsigma_{_\cT}^{bdm}, \, \nabla u_{_{\cT}} \right)_\sK
	+ \left(A\nabla u_{_{\cT}},  \, \nabla u_{_{\cT}} \right)_\sK \right)^{1/2}.
\eeq
The third term above is given in the previous section, and the other two terms may be computed by
\begin{eqnarray*}
	(A^{-1}\tilde \bsigma_{_\cT},\,  \tilde\bsigma_{_\cT})_\sK
	&=&
		\sum_{F \subset \partial K} \sum_{F' \subset \partial K} |F| \, |F'| 
		\left( B_{_{FF'}}-D_{_{FF'}}+M_{_{FF'}}\right) \\[2mm]
	\mbox{and} \quad
	 (\hat\bsigma_{_\cT},  \nabla u_{_\cT})_\sK
	 &=& -
	 \sum_{F \subset \partial K}  \dfrac{1}{12\, |K|}\, u_{_\cT}(\bx_\sF)\, \mbox{sign}_\sK (F)\,
	 |F|  \, \left(\bn_\sF^t {\bf L}_{_{F'}} \right).
\end{eqnarray*}
Here, the $B_{_{FF'}}$, $D_{_{FF'}}$, $M_{_{FF'}}$, and $ {\bf L}_{_{F'}}$ have the following formulas:
\begin{eqnarray*}
	B_{_{FF'}}&=&
	\hat \sigma^{bdm}_{s,\sF} \, \hat \sigma^{bdm}_{s,_{F'}} 
	\dfrac{(1+\delta_{\bs_\sF,\bs_{_{F'}}})}{48 \, |K|}
	 \left(A_K^{-1} \bt_{_{F_{\be}}},  \bt_{_{F'_{\be}}}\right)
	 \mbox{sign}_K(F_{\be})  \, \mbox{sign}_K(F'_{\be})|F_{\be}||F'_{\be}|,\\[4mm]
	D_{_{FF'}}&=&
		\hat \sigma^{bdm}_{s,\sF} \, \hat \sigma^{bdm}_{e,_{F'}}
		\dfrac{(1+\delta_{\bs_\sF,\be_{_{F'}}})}{48 \, |K|}
		 \left(A_K^{-1} \bt_{_{F_{\be}}},  \bt_{_{F'_{\bs}}}\right)
		\mbox{sign}_K(F_{\be}) \,\mbox{sign}_K(F'_{\bs})|F_{\be}||F'_{\bs}|\\[4mm]
		&&+
		\hat \sigma^{bdm}_{e,\sF}\, \hat \sigma^{bdm}_{s,_{F'}}
		\dfrac{(1+\delta_{\be_\sF,\bs_{_{F'}}})}{48 \, |K|}
		\left( A_K^{-1} \bt_{_{F_{\bs}}},  \bt_{_{F'_{\be}}}\right)
		\mbox{sign}_K(F_{\bs})  \, \mbox{sign}_K(F'_{\be})|F_{\bs}||F'_{\be}|,\\[4mm]
	M_{_{FF'}}&=&
		\hat \sigma^{bdm}_{e,\sF}  \, \hat \sigma^{bdm}_{e,_{F'}}
		\dfrac{(1+\delta_{\be_\sF,\be_{_{F'}}})}{48 \, |K|} 
		\left( A_K^{-1} \bt_{_{F_{\bs}}},  \bt_{_{F'_{\bs}}}\right)
		\mbox{sign}_K(F_{\bs}) \,  \mbox{sign}_K(F'_{\bs})|F_{\bs}||F'_{\bs}|,\\[4mm]
 \mbox{and} &&
    {\bf L}_{_{F'}}=
	 \sum_{F' \subset \partial K}
 		|F'|\left(
		\hat \sigma^{bdm}_{s,_{F'}} \,  \mbox{sign}_K(F'_{\be})|F'_{\be}| \bt_{_{F'_\be}}
		-\hat \sigma^{bdm}_{e,_{F'}}  \,  \mbox{sign}_K(F'_{\bs})|F'_{\bs}| \bt_{_{F'_\bs}} \right).
\end{eqnarray*}

\section{Numerical Experiments}
\setcounter{equation}{0}

In this section, we report some numerical results for the Kellogg benchmark test problem \cite{Kel:74}.
Let $\O=(-1,1)^2$ and
 \[
 u(r,\theta)=r^{\gamma}\mu(\theta)
 \]
in the polar coordinates at the origin with $\mu(\theta)$ being a
smooth function of $\theta$. The function
$u(r,\theta)$ satisfies the diffusion equation in (\ref{pde}) with $A= \a I$,
$\Gamma_N=\emptyset$, $f=0$, and
 \[
 \a(x)=\left\{\begin{array}{ll}
 R & \quad\mbox{in }\, (0,1)^2\cup (-1,0)^2,\\[2mm]
 1 & \quad\mbox{in }\,\O\setminus ([0,1]^2\cup [-1, 0]^2).
 \end{array}\right.
 \]
The $\gamma$ depends on the size of the jump.
In the test problem, $\gamma=0.1$ is chosen and is corresponding to
$R\approx 161.4476387975881$.
Note that the solution $u(r,\theta)$ is only in
$H^{1+\gamma-\epsilon}(\O)$ for some $\epsilon>0$ and, hence, it is
very singular for small $\gamma$ at the origin. This suggests that
refinement is centered around the origin.

This problem is tested by the standard $\ZZ$ estimator and its variation:
$$
 \xi_{_{\ZZ}} = \|\nabla u_{_\cT} - G(\nabla u_{_\cT})\|_{0,\Omega}
  \quad\mbox{and}\quad 
  \tilde{\xi}_{_{\ZZ}} = \|\a^{1/2}\nabla u_{_\cT} - \a^{-1/2}G(\a\nabla u_{_\cT})\|_{0,\Omega}.
$$
Here, $\xi_{_{\ZZ}}$ is the standard $\ZZ$ estimator, i.e.,  the $L^2$ norm of
the difference between the numerical and recovered gradients; 
and  the $\tilde{\xi}_{_{\ZZ}}$ is a modified version, where the flux is recovered in $C^0$ continuous finite element space. 
Both versions of the $\ZZ$ estimators perform badly with many unnecessary over-refinements along the interfaces
(see Figures 3 and 4). 
\begin{figure}[ht]
    \hfill
    \begin{minipage}[!hbp]{0.48\linewidth}
    \label{Lshapemesh}
        \centering
        \includegraphics[width=0.99\textwidth,angle=0]{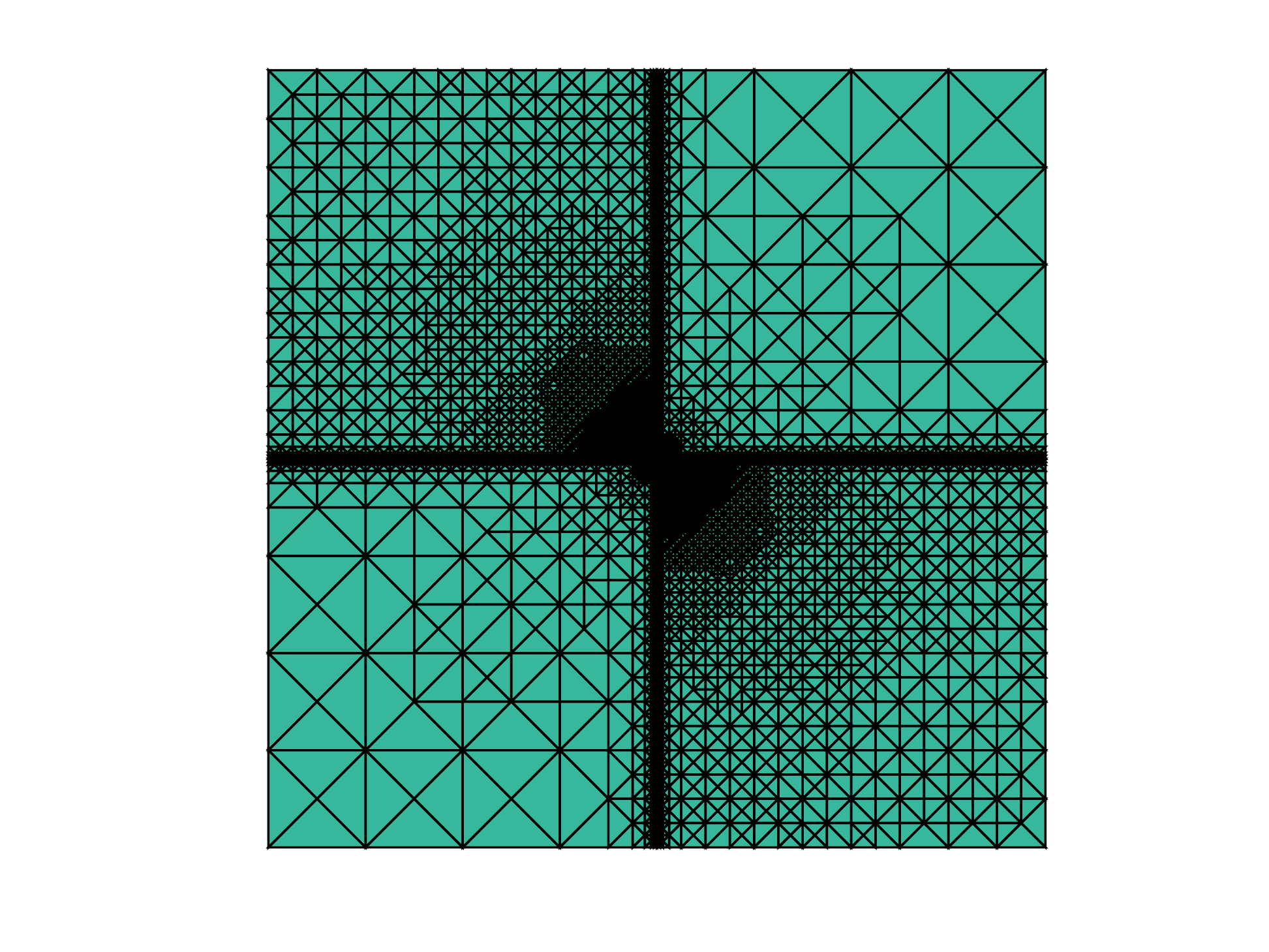}
        \caption{mesh generated by the $\ZZ$ indicatot $\xi_\sK^{_{\ZZ}}$
        }
        \end{minipage}%
        \quad
    \begin{minipage}[!htbp]{0.48\linewidth} \label{errorLshape}
        \centering
        \includegraphics[width=1\textwidth,angle=0]{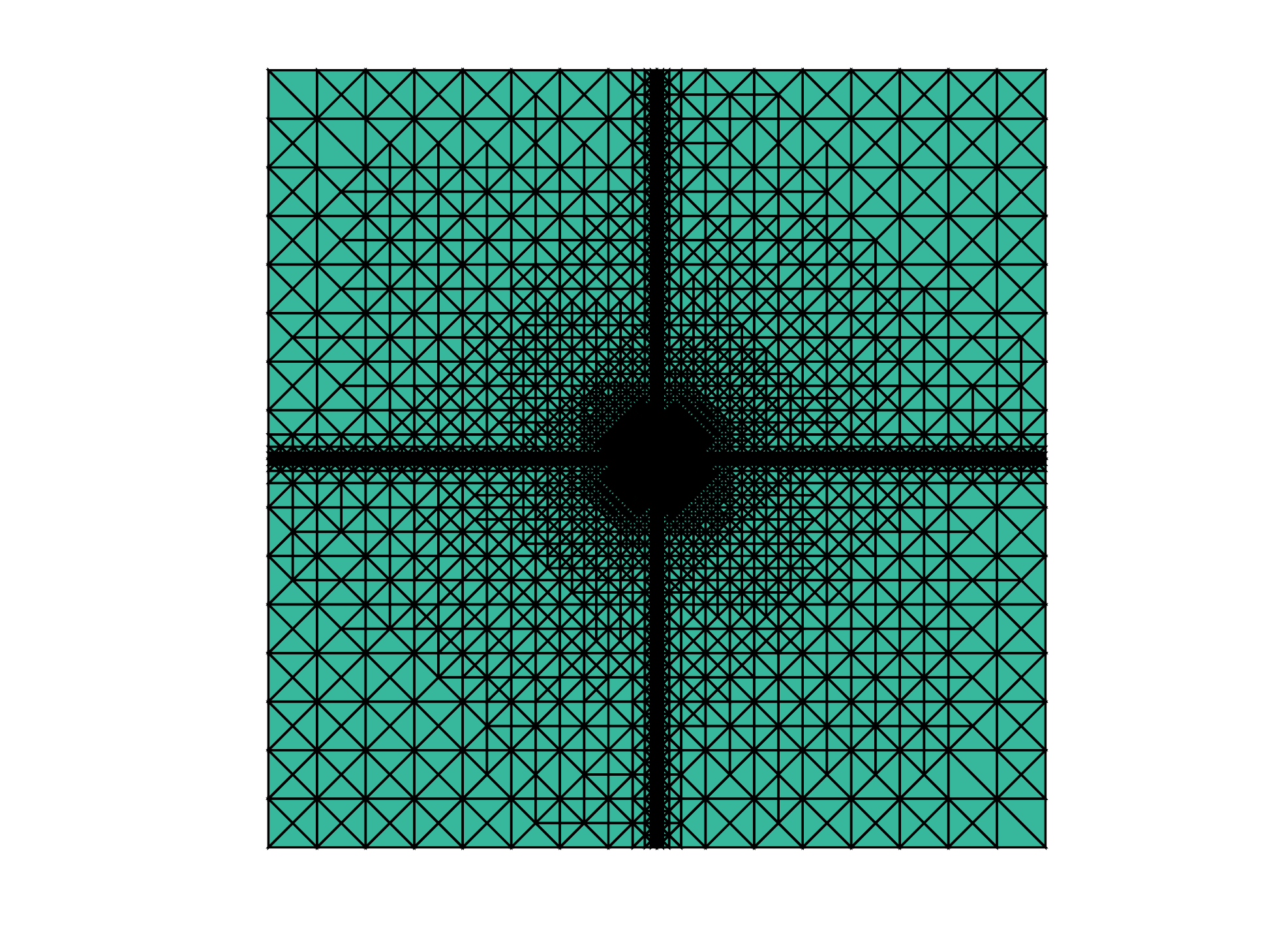}
        \caption{mesh generated by the modified $\ZZ$ indicator $\tilde{\xi}_\sK^{_{\ZZ}}$
        }
    \end{minipage}%
        \hfill
\end{figure}

\begin{figure}[ht]
    \begin{minipage}{0.48\linewidth} 
        \centering
        \includegraphics[width=0.99\textwidth,angle=0]{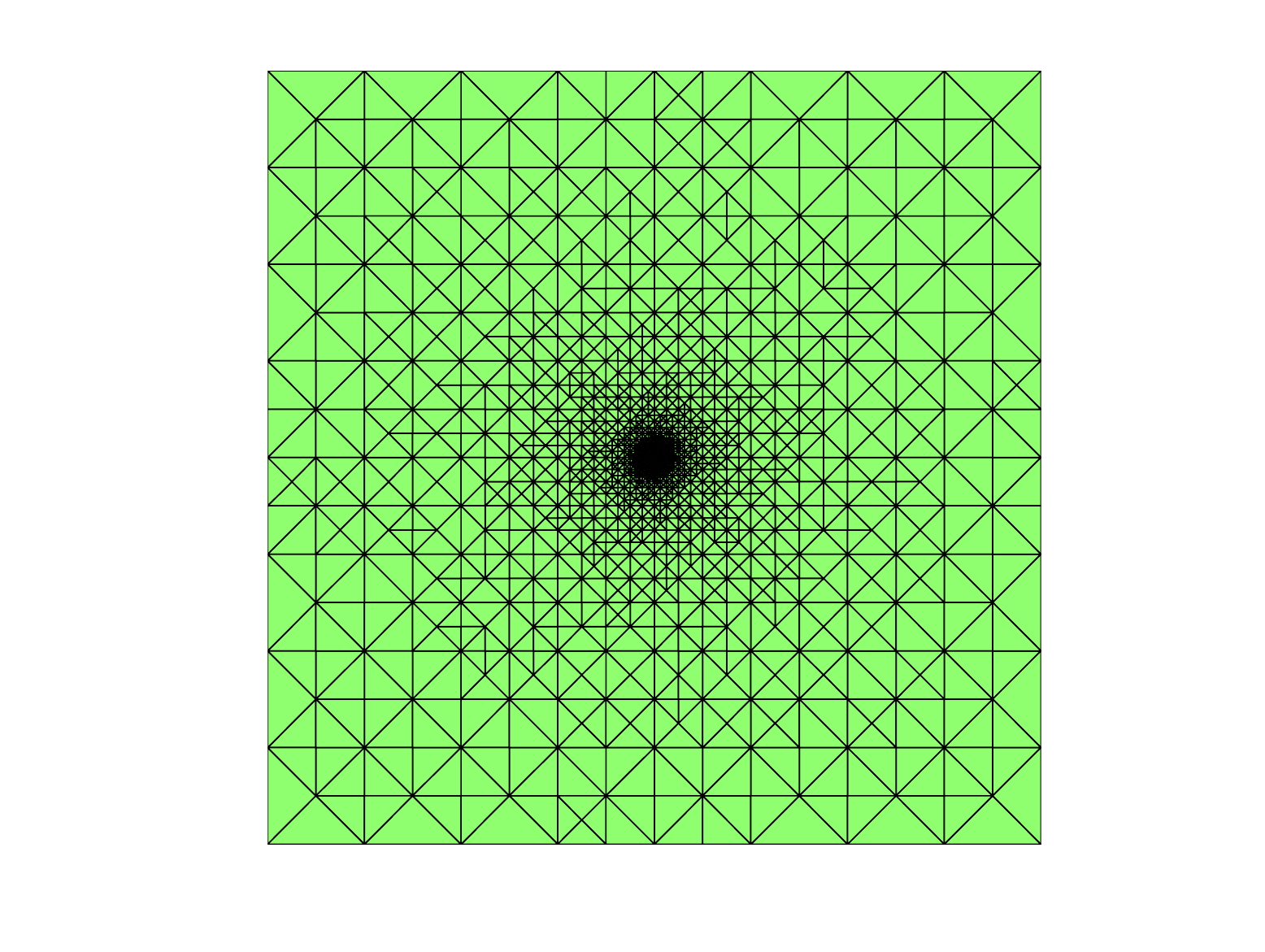}
        \caption{mesh generated by $\xi^{rt}_\sK$}%
        \label{meshRT_K}
        \end{minipage}%
        \quad
    \begin{minipage}{0.48\linewidth}        \centering
        \includegraphics[width=0.99\textwidth,angle=0]{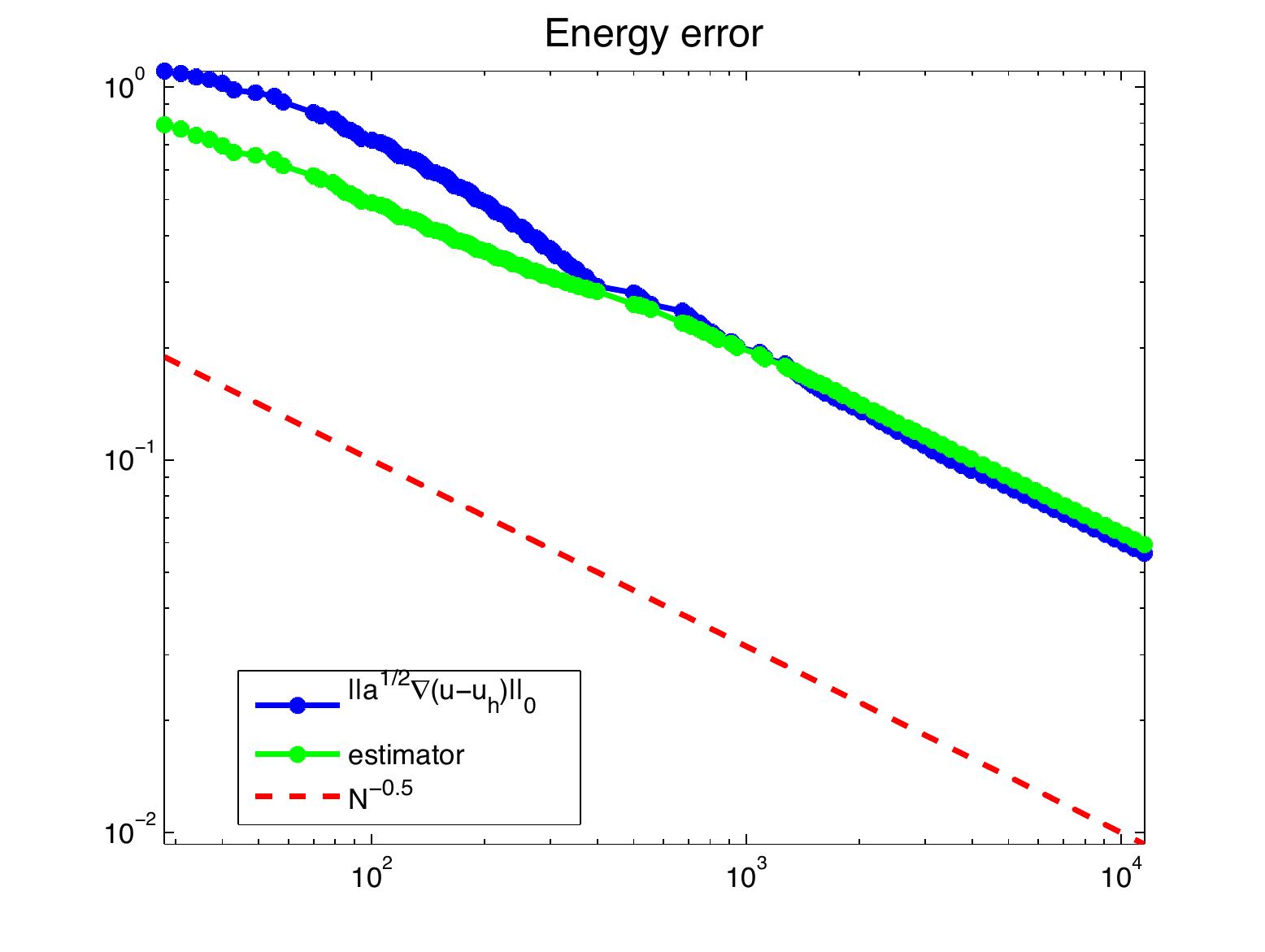}
        \caption{error and estimator $\xi^{rt}$}%
         \label{errorRT_K}
    \end{minipage}%
\end{figure}

\begin{figure}[ht]
    \begin{minipage}{0.48\linewidth}
        \includegraphics[width=0.99\textwidth,angle=0]{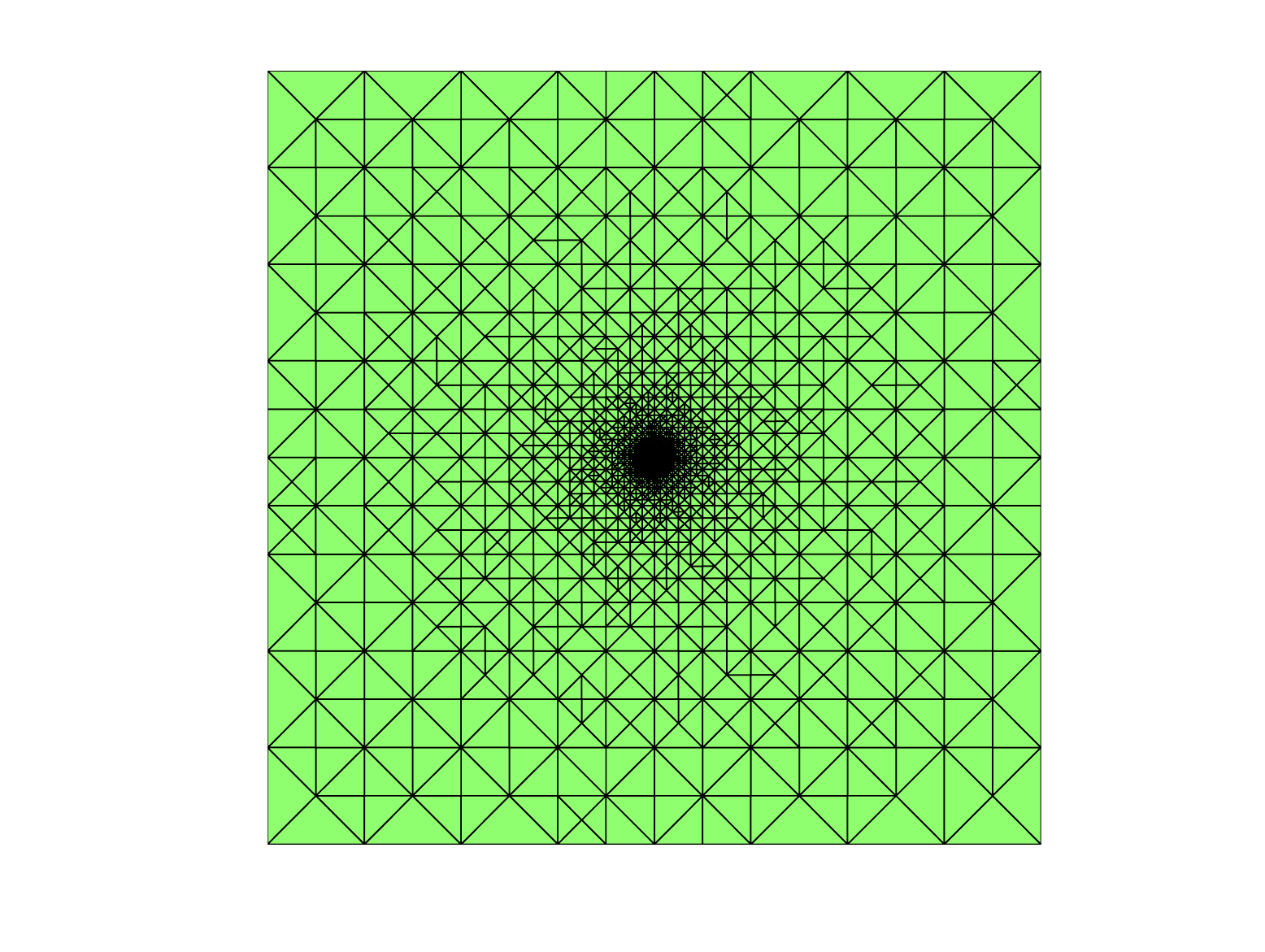}
        \caption{mesh generated by $\xi^{bdm}_\sK$}
        \label{meshBDM_K}
        \end{minipage}
        \quad
    \begin{minipage}{0.48\linewidth}
        \includegraphics[width=0.99\textwidth,angle=0]{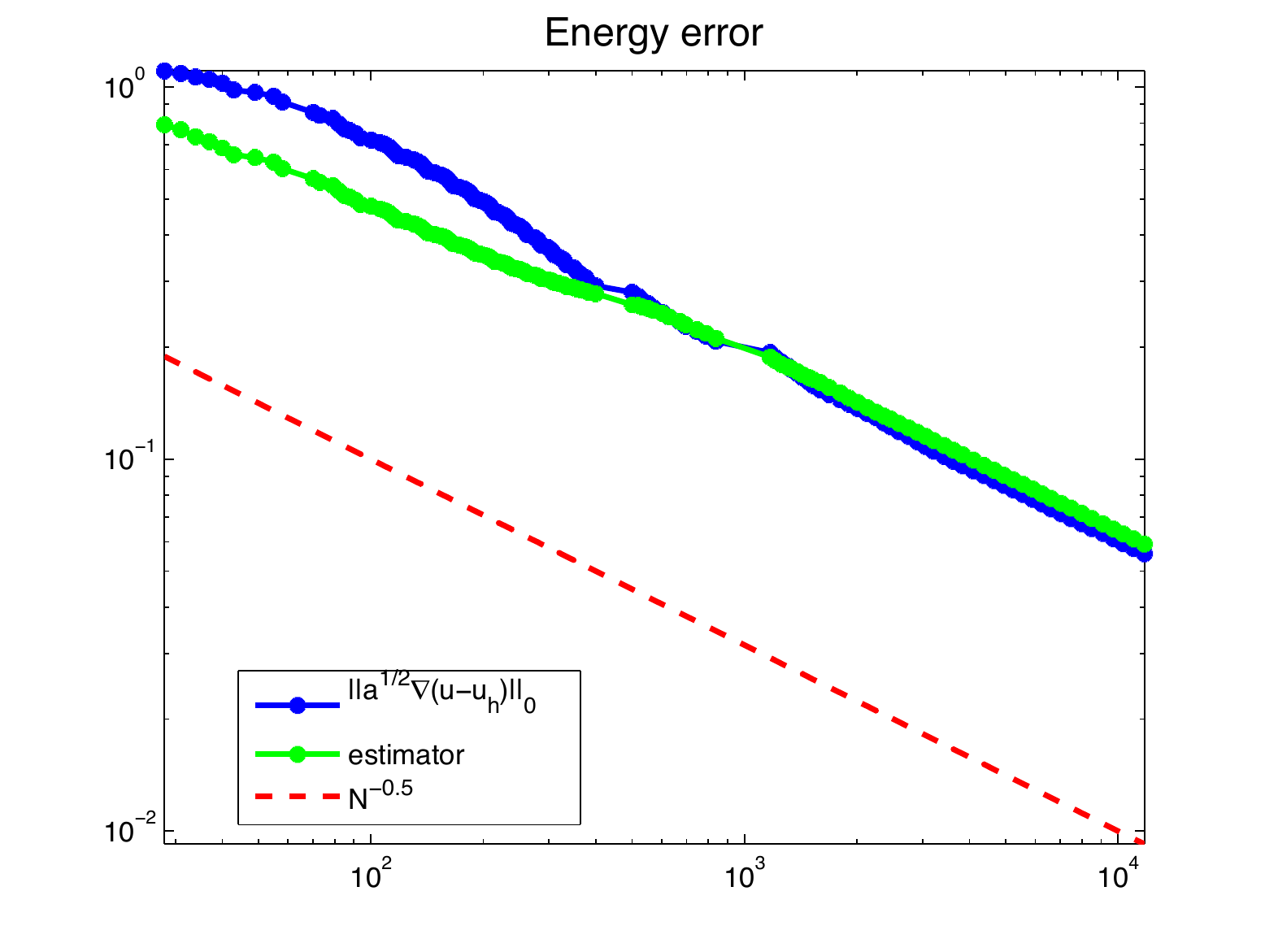}
        \caption{error and estimator $\xi^{bdm}$}
         \label{errorBDM_K}
    \end{minipage}
\end{figure}

\begin{figure}[ht]
    \begin{minipage}{0.48\linewidth} 
        \centering
        \includegraphics[width=0.99\textwidth,angle=0]{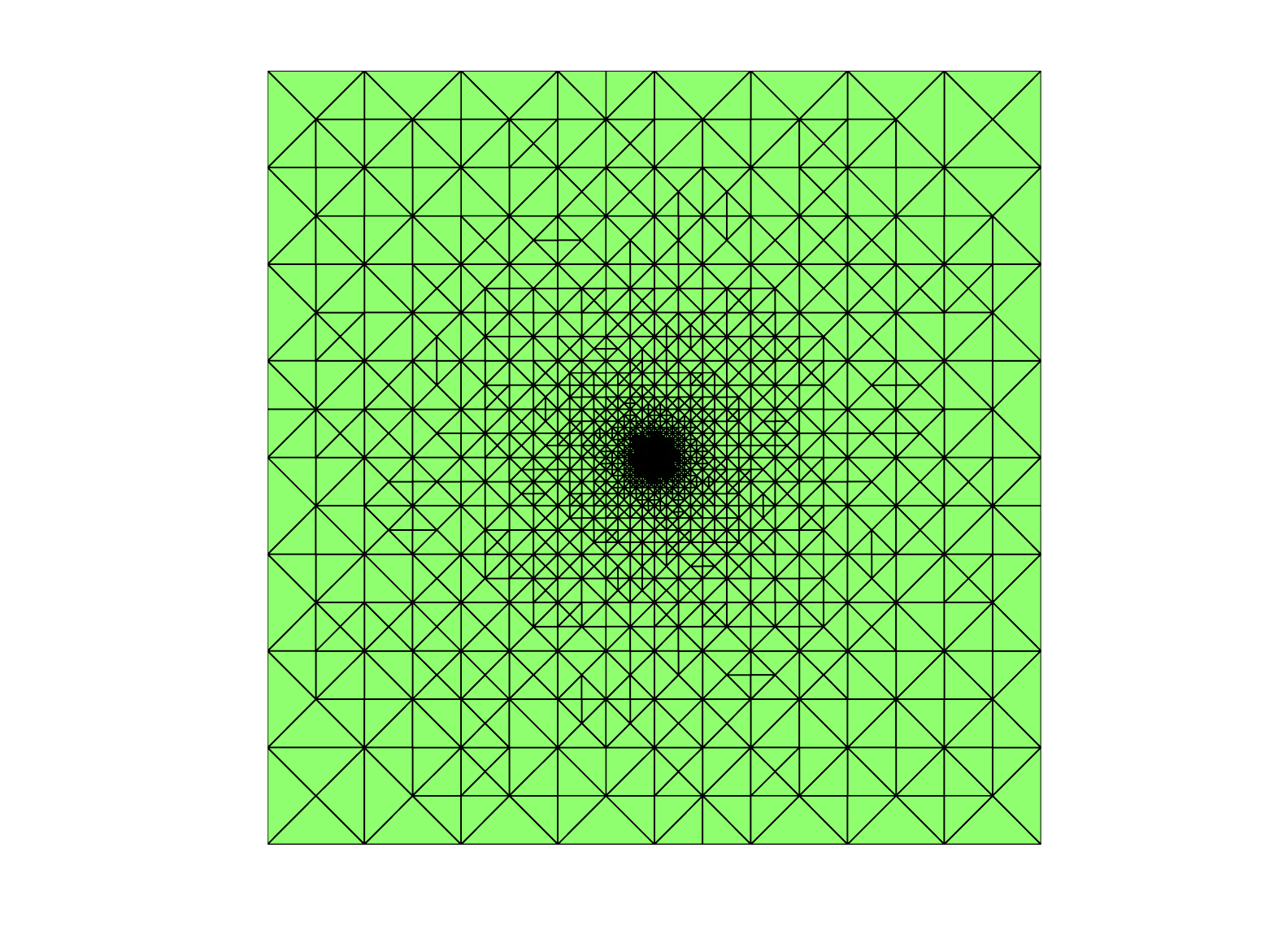}
        \caption{mesh generated by $\xi^{rt}_\sF$}%
        \label{meshRT_F}
        \end{minipage}%
        \quad
    \begin{minipage}{0.48\linewidth}        \centering
        \includegraphics[width=0.99\textwidth,angle=0]{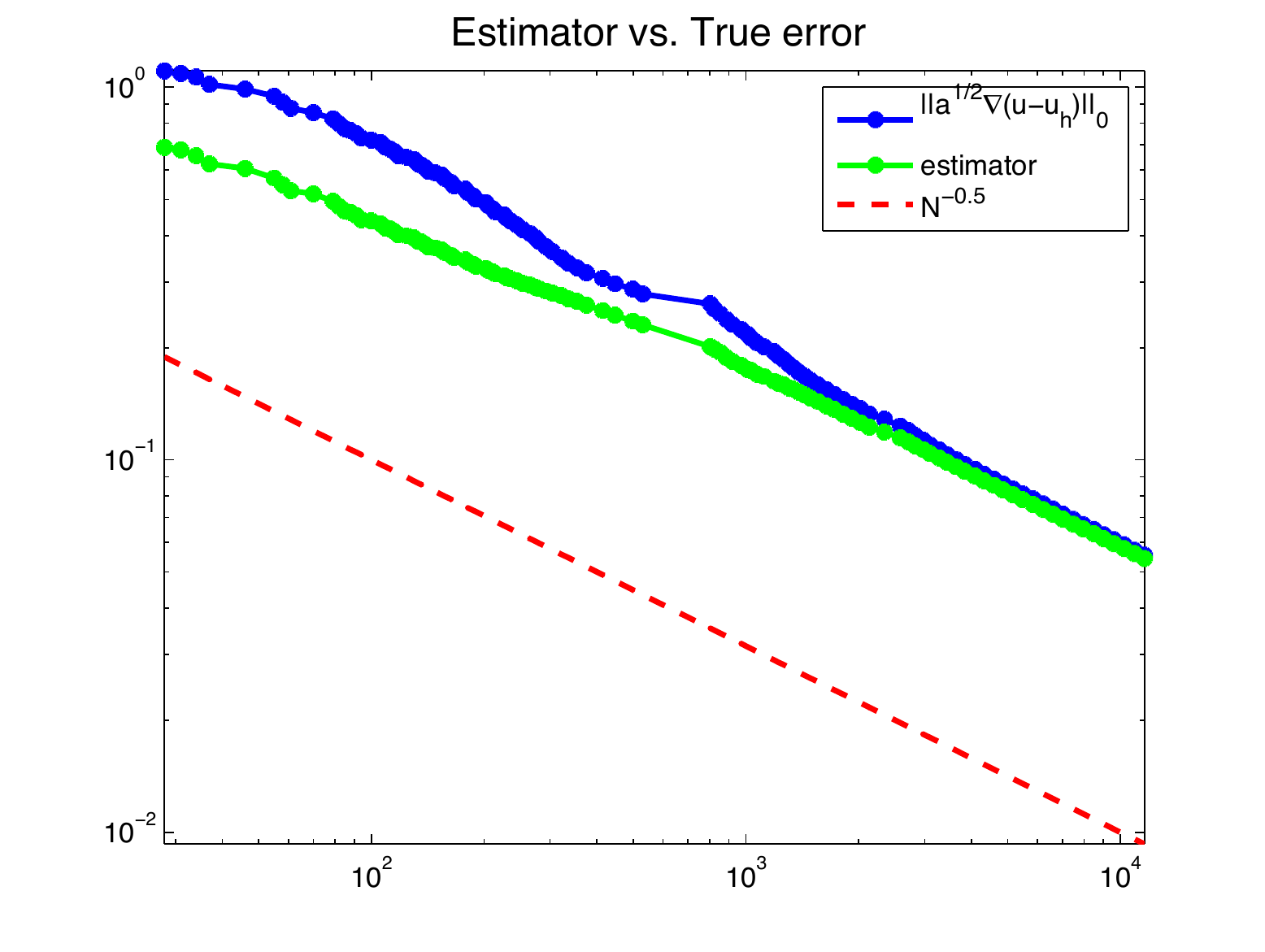}
        \caption{error and estimator $\hat \xi^{rt}$}%
         \label{errorRT_F}
    \end{minipage}%
\end{figure}
\begin{figure}[ht]
    \begin{minipage}{0.48\linewidth}
        \includegraphics[width=0.99\textwidth,angle=0]{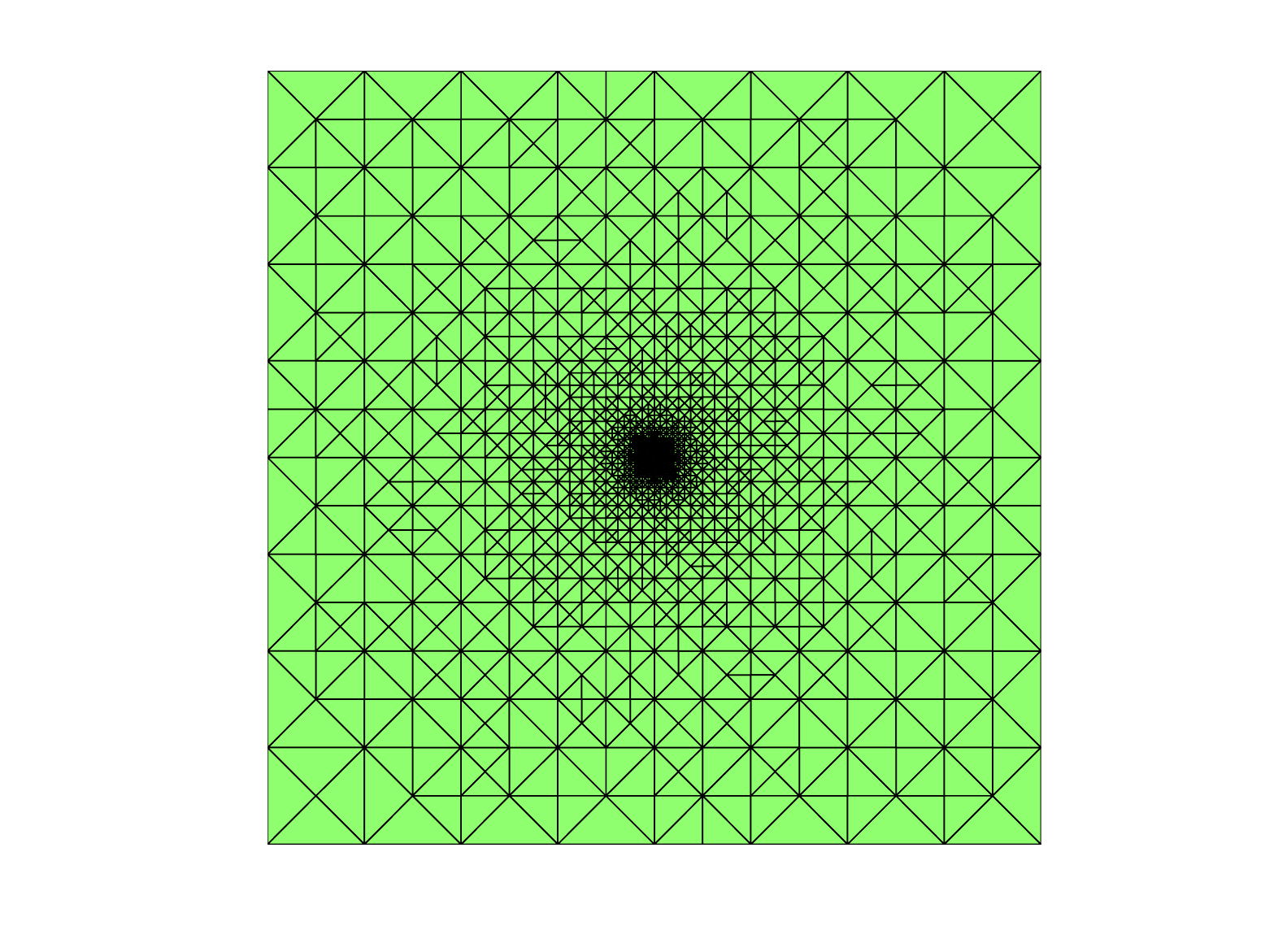}
        \caption{mesh generated by $\xi^{bdm}_\sF$}
        \label{meshBDM_F}
        \end{minipage}
        \quad
    \begin{minipage}{0.48\linewidth}
        \includegraphics[width=0.99\textwidth,angle=0]{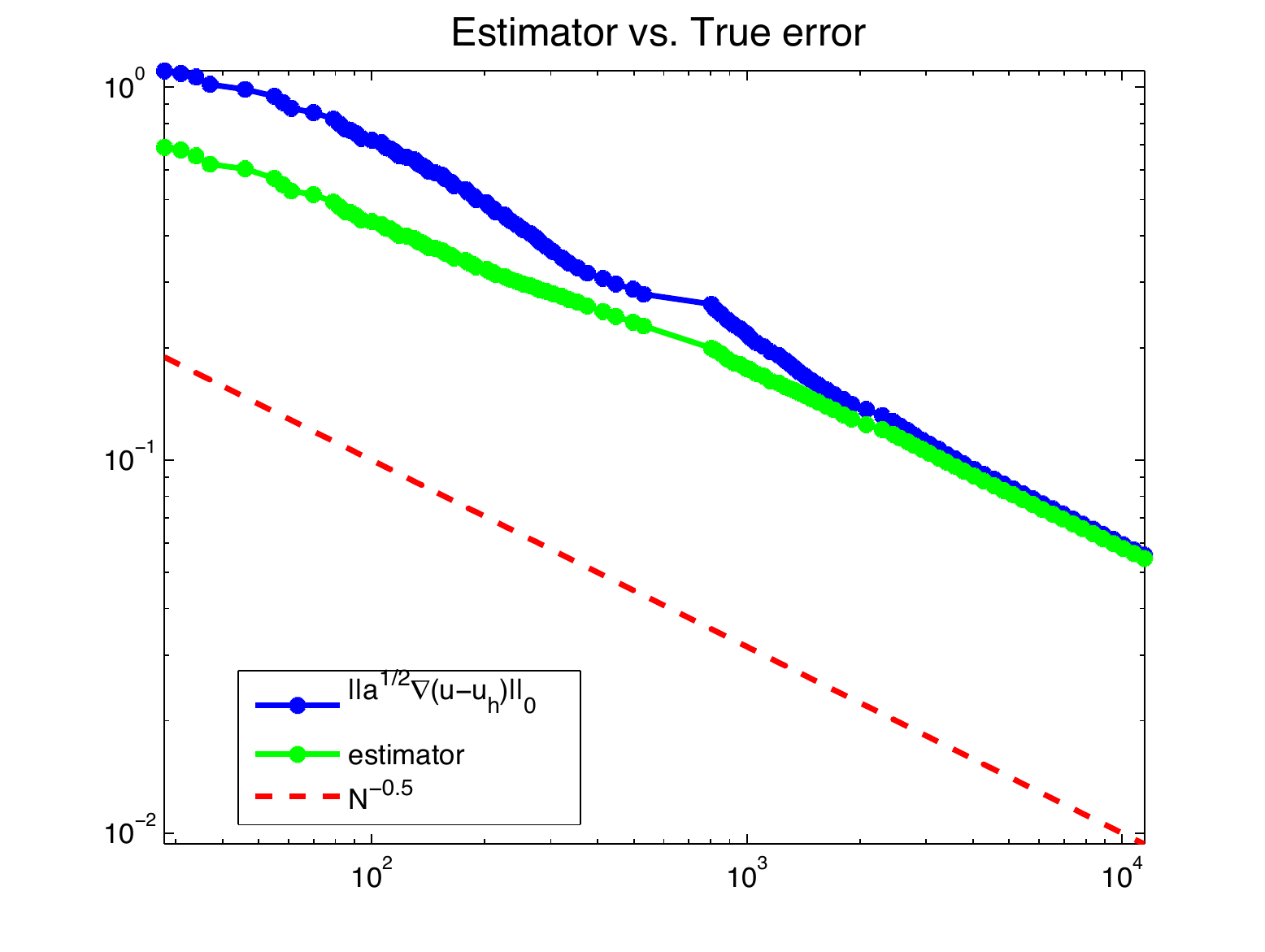}
        \caption{error and estimator $\hat \xi^{bdm}$}
         \label{errorBDM_F}
    \end{minipage}
\end{figure}

Meshes generated by $\xi_\sK$ and $\xi_\sF$ for both $RT$ and $\BDM$ recovery are shown in Figures \ref{meshRT_K},\ref{meshBDM_K}, \ref{meshRT_F} and \ref{meshBDM_F} , respectively. The refinements are centered at the origin and there is no over-refinements along the interfaces. 
Similar meshes for this test problem generated by other error estimators can be found in \cite{CaZh:09,CaZh:10a,CaZh:12}. 
The comparisons between  
the true error in the energy norm
and the estimators, $\xi$ and $\hat \xi$, are shown in Figures \ref{errorRT_K}, \ref{errorBDM_K}, \ref{errorRT_F}, and \ref{errorBDM_F}, respectively. 
All the estimators have effectivity indexes very close to one. Here, the effectivity index is defined as the ratio of the 
estimator and the true error in the energy norm.  
Moreover, the slope of the log(dof)- log (the relative error) for both $\xi$ and $\hat \xi$ are very close to $-1/2$, which indicates the optimal decay of the error with respect to the number of unknowns.

\end{document}